	\RequirePackage{fix-cm}
	\documentclass[smallextended]{svjour3}       % onecolumn (second format)
	\smartqed  % flush right qed marks, e.g. at end of proof
	\usepackage[linesnumbered,boxed, longend]{algorithm2e}   %算法排版样式3
	\usepackage{algpseudocode}
	\usepackage{amsfonts}
	\usepackage{graphicx}
	
	\usepackage{tabularx,multirow}
	\usepackage{caption}
	\usepackage{amssymb}
	\usepackage{etex}
	\usepackage{array}
	\usepackage{booktabs} %这里是使用图表的
	\usepackage{amsmath}
	
	\DeclareMathOperator*{\argmin}{arg\,min}
	
	\SetKwInOut{Input}{Input}
	\SetKwInOut{Output}{Output}
	
	\usepackage{hyperref}
	\hypersetup{
		colorlinks=true,
		linkcolor=blue,
		filecolor=magenta,
		urlcolor=cyan,
		citecolor=blue
	}
	\renewcommand\thesubsection{\bf	\thesection.\arabic{subsection}}
	\usepackage[square,comma,numbers,sort&compress]{natbib}
	\renewenvironment{proof}{{\emph{Proof. }}}{}
	\usepackage{url}
	\usepackage{subfigure}
	
	\newtheorem{mypro}{Proposition}

	\newtheorem{myremark}{Remark}
	
	\newtheorem{myclaim}{Claim}
	
	\newtheorem{mylemma}{Lemma}
	\newtheorem{myobservation}{Observation}
	\usepackage{pstricks}
	\usepackage{pst-plot}
	\usepackage{pst-eps}
	\usepackage{pst-grad}
	\usepackage{pgfplots}
	\usepackage{setspace}
	\usepackage{mathtools, nccmath}
	\renewcommand\figurename{\textbf{Figure}}
	
	\def\rev#1{{{#1}}}

	\renewcommand\S{Sect. }
	
\begin{document}
	\bibliographystyle{spbasic}

	\title{Combinatorial separation algorithms for the continuous knapsack polyhedra with divisible capacities \thanks{This work was supported by the Chinese Natural Science Foundation
			(No. 11631013) and the National 973 Program of China (No. 2015CB856002).}
	}
	\titlerunning{Separation algorithms for divisible knapsack polyhedra}
	%\author[1]{Don Joe}
	%\author[2]{Smith K.}
	%\author[1]{Wanderer}
	%\author[1]{Static}
	%%\affil[1]{TeX.SX}
	%%\affil[2]{Both on a bus}
	\author{Wei-Kun Chen \and Yu-Hong Dai}
	
	\institute{
	LSEC, ICMSEC, AMSS, Chinese Academy of Sciences, Beijing, China \and School of Mathematical Sciences, University of Chinese Academy of Sciences, Beijing, China. {\sl Emails}: cwk@lsec.cc.ac.cn; dyh@lsec.cc.ac.cn (corresponding author)}
	
	%%
	%\institute{Wei-Kun Chen \and Yu-Hong Dai \at
	%	LSEC, ICMSEC, Academy of Mathematics and Systems Sciences, \\Chinese Academy of Sciences, Beijing, China \and School of Mathematical Sciences, University of Chinese Academy of Sciences, \\ Beijing, China \\
	%	\email{\{cwk,dyh\}@lsec.cc.ac.cn}           %  \\
	%}

	\date{\today}

	\maketitle
	
	\begin{abstract}
  It is important to design separation algorithms of low computational complexity in mixed integer programming. We study the separation problems of the two continuous knapsack polyhedra with divisible capacities. The two polyhedra are the convex hulls of the sets which consist of $ n $ nonnegative integer variables, one unbounded continuous, $ m $ bounded continuous variables, and one linear constraint in either $ \geq  $ or $ \leq $ form where the coefficients of integer variables are integer and divisible. Wolsey and Yaman (Math Program 156: 1--20, 2016) have shown that the polyhedra can be described by adding the two closely related families of partition inequalities. However, no combinatorial separation algorithm is known yet for the polyhedra. In this paper, for each polyhedron, we provide a combinatorial separation algorithm with the complexity of $ \mathcal{O}(nm+m\log m ) $. In particular, we solve the separation problem for the continuous $ \geq $-knapsack polyhedron by presenting a polynomial-time separation algorithm for the family of $ \geq  $-partition inequalities. For the continuous $ \leq $-knapsack polyhedron, we derive the complemented partition inequalities, which dominate the $ \leq $-partition inequalities. The continuous $ \leq $-knapsack polyhedron is completely described with the addition of the proposed inequalities. We again show that the family of the proposed inequalities can be separated in polynomial time.
	\end{abstract}

	\section{Introduction}
	
Given a polyhedron $P$ and a point $ \bar{x} $, the {\emph{separation problem}} aims to find a hyperplane to separate $ P $ and $ \bar{x} $ or prove that no such one exists. In modern mixed integer programming solvers, several separation problems are solved at each iteration of the cutting plane phase. Hence, it is important to design separation algorithms that should be of low computational complexity.
	
%	finding a fast separation algorithm is crucial for the performance of the cutting plane method.
	
%	the cutting plane phase of solving the mixed integer programming.
%	Cutting plane method, which was proposed in \cite{} more than fifty years ago, has become one of the indispensable ingredients for solving mixed integer programming. Given a family of valid inequalities, one of the key components  in cutting plane method is to find a fast separation algorithm such that a violated inequality (or a cut) can be found quickly if such a inequality exists.
	
	In this paper, we are concentrated on the separation problem of the two continuous knapsack polyhedra with divisible capacities.
	The two polyhedra are convex hulls of the two continuous knapsack sets which consist of $ n $ nonnegative integer variables, one unbounded continuous, $ m $ bounded continuous variables, and one linear constraint in either $ \geq  $ or $ \leq $ form, where the coefficient of integer variables are integer and divisible. Mathematically, letting $N = \{1,\ldots, n\} $ and $ M = \{1,\ldots,m\} $, the  continuous $ \geq $-knapsack set under consideration can be written as
	\begin{equation*}
	X(b) = \bigg\{ (x, s) \in \mathbb{Z}^{n}_+ \times \mathbb{R}^{m+1}_+ \, : \, s_0  +
	\sum_{j\in M} s_j+ \sum_{i\in N} a_i x_i \geq b, \ 0 \leq s_j \leq u_j, j\in M \bigg\}
	\end{equation*}
	and the  continuous $ \leq $-knapsack set can be written as %the continuous $ \leq $-knapsack set with divisible capacities
	\begin{equation*}
	Y(b) = \bigg\{ (x, y) \in \mathbb{Z}^{n}_+ \times \mathbb{R}^{m+1}_+ \, : \,
		 \sum_{i\in N} a_i x_i \leq b + y_0 + \sum_{j \in M}y_j, \ 0 \leq y_j \leq u_j, j\in M \bigg\},
	\end{equation*}
	where $ a_i $ for $ i \in N $, $ u_j $ for $ j \in M$, and $ b $ are
	positive integers and  $ 1 | a_1 | \cdots | a_n $, {i.e.}, $ a_{i+1} $ is an integer multiple of $ a_i $. The weights $a_1$, $\ldots, a_n$ are distinct
and $ a_1 \geq 2 $. The two sets arise as the
	relaxation of the feasible set in many mixed integer linear programming problems, see {e.g.} \cite{Magnanti1993, Pochet1995,Gunluk1999,Bienstock1996,Chopra1998} for some of them. This motivates the research community to study the associated polyhedral structures.

	Several works have been conducted to study the related $ \geq $-knapsack set $ X'(b)= X(b) \cap \{ (x,s)\, : \, s_0 =0 \}$. \citet{Magnanti1993} first studied the case that $ n=1 $ and $ b = \sum_{j \in M} u_j $, which is corresponding to \emph{single facility splittable flow arc set}. In this case, they showed that the convex hull of the single facility splittable arc flow set is
completely described by the so-called \emph{residual capacity inequalities} and the initial constraints.
	%\citet{Atamturk2007} generalized this result for any $ b $.
	%\citet{Atamturk2002c} gave a linear time separation algorithm for the residual
	%capacity inequalities.
	\citet{Magnanti1995} studied the \emph{two facility splittable flow arc set} when $ n=2 $. They claimed without proof that the residual capacity inequalities
	and the initial constraints suffice to describe the convex hull as well. \citet{Pochet1995} considered the integer $ \geq $-knapsack set when $ m=0 $.
	%\begin{equation*}
	%Z(b) := \bigg \{x \in \mathbb{Z}^{n+1}_+ \, : \, s_0 + \sum_{i\in N} a_i x_i  \geq b\bigg \},
	%\end{equation*}
	%where $ a_i > 1$, $ i=1,\ldots,n $, are distinct and positive integers and $ a_1|a_2|\cdots|a_n $.
	They derived the {\it partition inequalities} and proved that these inequalities suffice to define the integer $ \geq $-knapsack polyhedron together with the nonnegative constraints. Furthermore, they gave a linear-time separation algorithm for the partition inequalities. \citet{Kianfar2012} showed that the partition inequalities can be obtained by
	applying the {\emph{mixed integer rounding}} \cite{Nemhauser1990,Marchand2001} procedure recursively.
	
	For the integer $ \leq $-knapsack set $  \big \{ x \in \mathbb{Z}_+^{n} \, : \, \sum_{i \in N}a_i x_i \leq b  \big \} $, \citet{Marcotte1985} provided a complete description for its convex hull. See  \citet{Pochet1998} for a generalization with bounded integer variables.

	Recently, \citet{Wolsey2016} considered the general $ X(b) $ and $ Y(b) $
	with arbitrary numbers of continuous variables and integer variables. They generalized the results of \cite{Pochet1995} and gave two closely related families of partition inequalities, i.e., $ \geq $- and $ \leq  $-partition inequalities. Furthermore, they proved that the convex hulls of $ X(b) $ and $ Y(b) $ are completely described with the addition of these $ \leq $-partition and $ \geq  $-partition inequalities, respectively.
	Let $ X_L(b) $ and $ Y_L(b) $ be the linear relaxation of $ X(b) $ and $ Y(b) $ obtained by dropping the integrality restriction on the variables $ x $, respectively.
	Their result implies that for each continuous knapsack polyhedron, the associated separation problem can be reduced to, given a point $ (\bar{x},\bar{s}) \in X_L(b) (Y_L(b)) $, either finding a violated partition inequality or proving no such one exists.
  Since optimizing a linear function over the two considered knapsack sets can be done in polynomial time \cite{Wolsey2016}, and since the corresponding optimization and separation problems for a polyhedron are polynomially equivalent \cite{Grotschel1981}, it is natural to hope that there exists an efficient combinatorial separation algorithm for each family of partition inequalities. However, no polynomial-time combinatorial separation algorithms was known yet for any of the two families of partition inequalities. For the family of $ \geq $-partition inequalities, this unsolved problem was explicitly mentioned in \cite{Wolsey2016}.
	\begin{quote}
		``Though polynomial time combinatorial separation algorithms are known
		both for the partition inequalities for the integer $ \geq $-knapsack
		set and the residual capacity inequalities for the single facility
		splittable flow arc set (see Pochet and Wolsey \cite{Pochet1995} and
		Atamt\"{u}rk and Rajan \cite{Atamturk2002c}, respectively), we do not know
		such an efficient combinatorial algorithm to separate the exponential
		family of partition inequalities."
	\end{quote}

For the above two separation problems of $ \text{conv}(X(b)) $ and $ \text{conv}(Y(b)) $, this paper shall provide
two polynomial-time combinatorial algorithms with the complexity of $ \mathcal{O}(mn + m\log m)$. In particular, we show that the problem of separating the family of $ \geq $-partition inequalities for $ X(b) $ can be reduced to $ m+1 $ problems
	each of which is to separate the partition inequalities for the associated integer $ \geq $-knapsack set. These problems
can be solved in polynomial time by applying the algorithm in \citet{Pochet1995}. For $ \text{conv}(Y(b)) $, we give the {complemented partition inequality} derived by the complemented continuous $ \geq $-knapsack set which is isomorphic to the continuous $ \leq $-knapsack set. The complemented partition inequality dominates the $ \leq $-partition inequality in \cite{Wolsey2016}. We show that the family of {complemented partition inequalities} and the initial constraints suffice to describe $\text{conv}(Y(b))$ and its separation problem can be solved in polynomial time.
	
	For the continuous $ \geq $-knapsack polyhedron, the above result generalizes a result of \citet{Atamturk2002c} for the separation problem of $ \text{conv}(X(b)) $ with $ n=1 $, where they presented a polynomial-time separation algorithm for the residual capacity inequalities.
	
	This paper is organized as follows. \S \ref{section:reviewpartitionineq} reviews the partition inequalities for the integer knapsack and gives a separation algorithm which is modified from \cite{Pochet1995}. \S \ref{section:separation} presents a polynomial-time
	algorithm for separating the family of $ \geq $-partition inequalities. \S \ref{lessknapsackset} aims at finding a combinatorial separation algorithm for the $ \leq $-continuous knapsack polyhedron. Finally, \S \ref{section:conclusions} gives the conclusions and some possible future works.
	
	Throughout this paper, we denote $ w(C)  = \sum_{i \in C} w_i$ for a vector $ w\in \mathbb{R}^n $ and a subset $ C\subseteq N $ unless specified. For  convenience, we define $ a_0 =1 $ and $ a_{n+1} = + \infty $. Also denote $ N_0 = N \cup \{0\} $, i.e., $ N_0 = \{0,1\ldots, n\} $. Given two integers $ b $ and $ c $,  $ c \nmid b $ if and only if $ b $ is not a multiple of $ c $.

	\section{An improved separation algorithm for the partition inequalites of the integer $\ge$-knapsack set}
	\label{section:reviewpartitionineq}
	In this section, we consider the partition inequalities for the integer $ \geq $-knapsack set with divisible capacities
	\begin{equation*}
	Z(b) = \bigg \{x \in \mathbb{Z}^{n+1}_+ \, : \, x_0 + \sum_{i\in N} a_i x_i  \geq b\bigg \},
	\end{equation*}
	where $ 1 | a_1 \cdots | a_n $. Similarly, denote $ Z_L(b) $ be the linear relaxation of $ Z(b) $.
	We first review the partition inequalities derived by \citet{Pochet1995}, which is shown to described $ \text{conv}(Z(b)) $
    together with the initial constraints, and the separation algorithm presented in \cite{Pochet1995}. Then an improved separation algorithm is provided.
			
	To begin with, we introduce the partition inequalities for $ Z(b) $.
	Let
	\begin{equation}
	\label{partition}
	\begin{array}{ll}
	\Pi= \big \{ & \{i_1=0,i_{1} +1 ,\ldots,j_1\}, \\
	&\{i_2=j_1+1,i_2+1,\ldots,j_2\}, \\
	&\qquad \vdots\\
	&\{i_p = j_{p-1}+1,i_p+1,\ldots,j_p = n \} \ \ \big\},
	\end{array}
	%	\{i_1=1,i_{i_1+1},\ldots,j_1\}, \{i_2=j_1+1,i_2+1,\ldots,j_2\}, \ldots,
	%	\{i_p = j_{p-1}+1,i_p+1,\ldots,j_p = n \}
	\end{equation}
	be a partition of $ N_0 $, where $ p \in \mathbb{Z}_{++} $ and $ a_{i_p} \leq b  $. Define $ \beta_{p} =b$ and
	\begin{equation}
	\label{paradef}
	\kappa_t = \bigg \lceil \frac{\beta_t}{a_{i_t}} \bigg \rceil, \ \mu_t = (\kappa_t -1)a_{i_t}, \ \beta_{t-1} = \beta_t - \mu_t, \ \text{for} \   t=p, \ldots,1.
	\end{equation}
	Notice that $ a_{i_1} = a_0 = 1 $.
	The partition inequality for $ Z(b) $ given by \cite{Pochet1995} is
	\begin{equation}
	\label{intpartitionineq}
	x_0 + \sum_{i=1}^{j_1}\min\{a_i, \kappa_1\}x_i +
	\sum_{t=2}^p \prod_{l=1}^{t-1}\kappa_l \sum_{i=i_t}^{j_t} \min\bigg \{ \frac{a_i}{a_{i_t}}, \kappa_t \bigg \} x_i \geq \prod_{t=1}^{p}\kappa_t,
	\end{equation}
	which is valid for $ Z(b) $. For convenience, for some specific partition $ \Pi $, we \rev{refer} the corresponding inequality \eqref{intpartitionineq} as the $ \Pi $-partition inequality. In particular, when the partition is fixed as $ \{ N_0\} $, \eqref{intpartitionineq} reduces to
	\begin{equation}\label{npartition}
	x_0 + \sum_{i=1}^r a_i x_i + b \sum_{i=r+1}^n x_i \geq b,
	\end{equation}
	where $ r \in N_0 $ satisfying $ a_r \leq b < a_{r+1} $. Notice that $ r $ can be equal to $ n $ and $ a_{n+1} = + \infty $. For simplicity, we just call \eqref{npartition} the $ N_0 $-partition inequality.
	
	\citet{Pochet1995} gave a decomposition approach to separate the family of partition inequalities \eqref{intpartitionineq}. To describe it, given a point $ \bar{x} \in  Z_L(b) $, let
	\begin{equation}
	\label{intparsepaineq}
	\delta = \sum_{i=r+1}^n \bar{x}_i,  \  \kappa = \bigg \lfloor \frac{b}{a_{r}} \bigg \rfloor +1, \ \text{and} \  \omega = (\kappa -1)a_{r}.
	\end{equation}
	In the case that $ a_r | b $, $ \delta \geq 1 $, or $ \bar{x}_0+ \sum_{i=1}^{r}a_i \bar{x}_i  < \omega(1-\delta)  $, Pochet and Wolsey \cite{Pochet1995} established the following result.
	\begin{theorem}%[\citet*{Pochet1995}]
		\label{trivialresult}
		Given a point $ \bar{x} \in  Z_L(b)$, let $ \delta $,  $ \kappa $, and $  \omega $ defined as in \eqref{intparsepaineq}.
		\begin{itemize}
			\item [{\rm(i)}] If $ a_{r} | b $,  then $\bar{x} \in \text{conv}(Z(b))$ if and only if the $ N_0 $-partition inequality \eqref{npartition} is satisfied by $ \bar{x} $;
			%$ \bar{x}_0 + \sum_{i=1}^{r}a_i \bar{x}_i + b \sum_{i=r+1}^n \bar{x}_i \geq b $;
			\item [{\rm(ii)}] If $ \delta \geq 1 $, then $ \bar{x} \in \text{conv}(Z(b)) $;
			\item [{\rm(iii)}] If $ \bar{x}_0 + \sum_{i=1}^{r}a_i \bar{x}_i  < \omega(1-\delta)  $, then the $ N_0 $-partition inequality \eqref{npartition} is a most violated inequality.
		\end{itemize}
	\end{theorem}
	\noindent The remaining case is then the following,
	\begin{equation}
	\label{nontrivialcase}
	a_{r} \nmid b , \  \delta < 1, 	\ \text{and} \ \bar{x}_0 + \sum_{i=1}^{r}a_i \bar{x}_i \geq \omega(1-\delta) .
	\end{equation}
	In this case, let $ v \leq r $ be the index such that
	\begin{equation}
	\label{vcond}
	 \sum_{i=v+1}^r a_i\bar{x}_i < \omega(1-\delta) \ \text{and} \  \sum_{i=v}^r a_i\bar{x}_i \geq \omega(1-\delta).
	\end{equation}
	Notice that $ a_0 = 1 $ and $ v $ can be $ 0 $.
	Define the vectors $ \alpha $ and $ \gamma $ as
	\begin{equation}
	\label{alphagammadef}
	\alpha_i = \left\{
	\begin{array}{ll}
	0,  &   i<v;\\
	\frac{\omega (1-\delta)-\sum_{i=v+1}^{r}a_i \bar{x}_i}{a_v}, & i=v;\\
	\bar{x}_i, & i > v
	\end{array}
	\right. \text{and} \ \
	\gamma_i = \left\{
	\begin{array}{ll}
	\bar{x}_i - \alpha_i,  &   i\leq r;\\
	\bar{x}_i, & i > r.
	\end{array}
	\right.
	\end{equation}
	By the above definition of $ \alpha $, the following result can easily be checked.
	\begin{myobservation}
		\label{omegaequality}
		Let $ \alpha $ be defined in \eqref{alphagammadef}. It follows that
		\begin{equation}
		\label{obsineq1}
		\sum_{i=v}^r a_i \alpha_i + \omega \sum_{i=r+1}^n \alpha_i = \omega.
		\end{equation}
	\end{myobservation}
	\noindent Under the decomposition of \eqref{alphagammadef}, Pochet and Wolsey \cite{Pochet1995} further established the following
 decomposition result.
	\begin{theorem}%[\citet*{Pochet1995}]
		\label{nontrivialresult}
		Given a point $ \bar{x}\in Z_L(b) $, let $ \delta $,  $ \kappa $, and $  \omega $ defined as in \eqref{intparsepaineq}. Consider the case \eqref{nontrivialcase} and define $ \alpha $ and $ \gamma $ as in \eqref{alphagammadef}. Then $ \bar{x} \in \text{conv}(Z(b)) $ if and only if $ \gamma \in \text{conv}(Z(b-\omega)) $. Moreover, for $ Z(b-\omega) $, if
		$ \gamma $ violates the $ \Pi $-partition inequality \eqref{intpartitionineq} by $ \epsilon $, then for $ Z(b) $,
		\begin{itemize}
			\item [\rm{(i)}]
			if $ a_{v} \leq b - \omega $, it follows that $ \bar{x} $ violates $ \Pi $-partition inequality \eqref{intpartitionineq} by $ \epsilon $;
			\item [\rm{(ii)}] if $ a_{v} >b - \omega $, it follows that $ i_p < v $ and
			$ \bar{x} $ violates the $ \Pi_1 $-partition inequality \eqref{intpartitionineq} by $ \epsilon $, where
			\begin{equation}
			\Pi_1 = \Pi \backslash \big \{\{ i_p, \ldots, n \} \big \} \cup \big \{ \{ i_p, \ldots, v-1 \} , \{ v, \ldots, n \} \big \}.
			\end{equation}
		\end{itemize}
	\end{theorem}
	
	In order to solve the separation problem of $ \text{conv}(Z(b)) $, \citet{Pochet1995} recursively used Theorem \ref{nontrivialresult} to get an equivalent separation problem until one of the cases (i)-(iii) in Theorem \ref{trivialresult} is satisfied. Therefore it suffices to design an algorithm to solve the separation problem of $ \text{conv}(Z(b))$ based on Theorems \ref{trivialresult} and \ref{nontrivialresult}.

  In the following, we shall focus on the case that $ v = 0 $ in \eqref{vcond} and provide an improved separation algorithm. As will be seen in Section \ref{section:separation}, this specific analysis plays an important role in analyzing the general set $X(b)$. To do so, we give a lemma.
	\begin{mylemma}
	 	Let $ r \in N_0 $ such that $ a_r \leq b < a_{r+1} $. Given a point $ \bar{x} \in Z_L(b) $, if $ \bar{x}_1 = \cdots = \bar{x}_{r} = 0 $, then $ \bar{x} \in \text{conv}(Z(b)) $ if and only if the $ N_0 $-partition inequality $ \eqref{npartition} $ holds at $ \bar{x} $.
	\end{mylemma}
	\begin{proof}
		Define a point $ (\hat{x}_0, \hat{x}_{r+1}, \ldots, \hat{x}_n) \in \mathbb{R}^{n+1-r}  $ by setting $ \hat{x}_0 = \bar{x}_0 $, and $ \hat{x}_{i} = \bar{x}_i$ for $ i = r+1, \ldots, n $.
		Since $ \bar{x}_1 = \cdots = \bar{x}_{r} = 0 $, $ \bar{x} \in \text{conv}(Z(b)) $ if and only if $ \hat{x} \in \text{conv}(Z'(b)),$ where
		$$
			Z'(b) = \bigg\{ (x_0,x_{r+1}, \ldots, x_{n}) \in \mathbb{Z}^{n+1-r}_+ \, : \, x_0 + \sum_{i=r+1}^n a_i x_i \geq b  \bigg\}.
		$$
		 As $ a_0 (=1) \, | \,  b  $ and $ a_{r+1} >b  $, it follows
		from item (i) of Theorem \ref{trivialresult} that $ \hat{x} \in \text{conv}(Z'(b)) $ if and only if
		\begin{equation}
			\label{lemtrivialproof}
			\hat{x}_0 + b\sum_{i=r+1}^n \hat{x}_i \geq b.
		\end{equation}
		Since $ \bar{x}_1 = \cdots = \bar{x}_{r} = 0 $, \eqref{lemtrivialproof} is equivalent to that \eqref{npartition} holds at the point $ \bar{x} $, which completes the proof. \qed
	\end{proof}
	
	\begin{mypro}
			\label{anothertrivialresult}
			Given a point $ \bar{x}\in Z_L(b) $, consider the case \eqref{nontrivialcase} and define $ v $ such that \eqref{vcond} is satisfied. If $ v = 0 $, then $ \bar{x} \in \text{conv}(Z(b)) $ if and only if the $ N_0 $-partition inequality \eqref{npartition} is satisfied by $ \bar{x} $. %Furthermore, if \eqref{npartition} is violated by $ \bar{x} $, \eqref{npartition} is one of the most violated partition inequalities for $ Z(b) $.
	\end{mypro}
	\begin{proof}
		By Theorem \ref{nontrivialresult}, $ \bar{x} \in \text{conv}(Z(b))  $ if and only if $ \gamma \in \text{conv}(Z(b-\omega)) $. Let $ r' \in N_0  $ such that $ a_{r'} \leq b- \omega < a_{r'+1} $. It follows from $ a_r \leq b < a_{r+1} $ that $ r' \leq r $. From the fact that $ v =0 $ and the definition of $ \gamma $ in \eqref{alphagammadef}, we have
		\begin{equation}\label{gammares1}
			 \gamma_1= \cdots = \gamma_{r'} =\gamma_{r'+1} = \cdots = \gamma_r =  0.
		\end{equation}
		Together with Lemma \ref{anothertrivialresult}, $ \bar{x} \in \text{conv}(Z(b)) $ is further equivalent to the fact that
		\begin{equation}
			\label{pro1ineqgamma}
			\gamma_0 + \sum_{i=1}^{r'} a_i\gamma_i  + (b-\omega)\sum_{i=r'+1}^n \gamma_i \geq b - \omega.
		\end{equation}
		By $v=0$ and the definitions of $ \alpha $ and $ \gamma $ in \eqref{alphagammadef},  we have
		\begin{equation}\label{gammaalphares1}
			\begin{array}{ll}
			 \bar{x}_0 = \alpha_0 + \gamma_0,  & \\
			\bar{x}_i = \alpha_{i},  &\qquad  \text{for} \ i=1, \ldots,r,\\
			\bar{x}_i = \alpha_{i} = \gamma_i, & \qquad \text{for} \ i=r+1, \ldots,n.
			\end{array}
		\end{equation}
		Adding \eqref{pro1ineqgamma} and \eqref{obsineq1}, the right hand side is $ b $ while the left hand side is
		 \begin{equation}
			\begin{aligned}
			& \ \gamma_0 + \sum_{i=1}^{r'} a_i\gamma_i  + (b-\omega)\sum_{i=r'+1}^n \gamma_i +
			\sum_{i=0}^r a_i \alpha_i + \omega \sum_{i=r+1}^n \alpha_i%  + b\sum_{i=r+1}^n \bar{x}_i
			\\
			= &\ \gamma_0 + (b-\omega) \sum_{i=r+1}^n \gamma_i + \sum_{i=0}^r a_i \alpha_i + \omega \sum_{i=r+1}^n \alpha_i  \qquad  (\text{from \eqref{gammares1}})\\
			= &\ (\gamma_0 + \alpha_0) + (b-\omega) \sum_{i=r+1}^n \gamma_i +\sum_{i=1}^r a_i \alpha_i + \omega \sum_{i=r+1}^n \alpha_i & \\
			= & \ \bar{x}_0 + \sum_{i=1}^r a_i \bar{x}_i + b \sum_{i=r+1}^n \bar{x}_i .  \qquad  (\text{from \eqref{gammaalphares1}})
			\end{aligned}
		 \end{equation}
		 This implies that $ \bar{x} \in \text{conv}(Z(b)) $ if and only if the $ N_0 $-partition inequality \eqref{npartition} is satisfied by $ \bar{x} $. \qed
	%	
	%	 We now proceed the se
	%	As $ a_r \nmid b $ in \eqref{nontrivialcase}, we have $ b - \omega = b - \big \lfloor \frac{b}{a_{r}} \big \rfloor \geq 1 $.
	\end{proof}
	
	Based on Theorems \ref{trivialresult} and \ref{nontrivialresult} and Proposition \ref{anothertrivialresult}, we describe a procedure of separating the partition inequality \eqref{intpartitionineq} in Algorithm \ref{intparseparalgo}.
	\IncMargin{0.5em}
	\begin{algorithm}[ht!]
	%	\SetNlSty{}{\bf S}{}
	%	\SetNlSkip{0.2em}
	%	\SetAlgoNlRelativeSize{-2}
		\caption{The separation algorithm for the partition inequalities \eqref{intpartitionineq} for $ Z(b) $}
		\label{intparseparalgo}
		\KwIn{The set $ Z(b) $ and the point $ \bar{x}  \in  Z_L(b)$}
		Initialize $ b^1 = b $, $ x^1 = \bar{x} $, $ h=1 $, $ \Pi = \{ \{ 1,\ldots,n \} \} $, and $ \bar{n} = n $\;
		\While{}
		{
			Compute $ r^h $ such that $ a_{r^h} \leq b^h < a_{r^h+1} $ and $ \delta^h $, $ \kappa^h $, $ \omega^h $ as in \eqref{intparsepaineq} \; \label{paralgo:step1}
			\If {$ a_{r^h} | b^h $}
			{
				\label{algostop:cond1}
				\eIf {$ x^h_0 + \sum_{i=1}^{r^h} a_i x^h_i + b^h \sum_{i=r^h+1}^n x^h_i < b^h $}
				{
					{\bf Stop} and construct the violated $ \Pi $-partition inequality \label{stopandfind1} \;
				}
				{
					{\bf Stop} and report that no violated inequality exists   \;
				}
			}
			\If {$ \delta^h \geq 1 $}
			{
				 {\bf Stop} and report that no violated inequality exists\;
			}
			\If {$x^h_0 + \sum_{i=1}^{r^h} a_i x^h_i + \omega^h \sum_{i=r^h+1}^n x^h_i < \omega^h$}
			{
				{\bf Stop} and construct the violated $ \Pi $-partition inequality \label{stopandfind2}\;
			}
	%		If $ a_{r} | b $, either go to {\bf S9} if $ \bar{x}_0 + \sum_{i=1}^r a_i \bar{x}_i + b \sum_{i=r+1}^n \bar{x}_i < b $ or stop and report that no violated inequality exists if $ \bar{x}_0 + \sum_{i=1}^r a_i \bar{x}_i + b \sum_{i=r+1}^n \bar{x}_i \geq b $\label{paralgo:stop1} \;
	%		If $ \delta \geq 1 $, stop and report that no violated inequality exists
	%		\label{paralgo:stop2}\;
	%		If $\bar{x}_0 + \sum_{i=1}^r a_i \bar{x}_i + \omega \sum_{i=r+1}^n \bar{x}_i < \omega $, go to {\bf S9} \label{paralgo:stop3}\;
			Compute $ v^h $ such that \eqref{vcond} \label{paralgo:keystep}\;
			\If { $ v^h = 0 $ }
			{ \label{paralgo:keystep1}
				\eIf {$ x^h_0 + \sum_{i=1}^{r^h} a_i x^h_i + b^h \sum_{i=r^h+1}^n x^h_i < b^h $}
				{
					{\bf Stop} and construct the violated $ \Pi $-partition inequality \label{stopandfind3} \;
				}
				{
					{\bf Stop} and report that no violated inequality exists \label{newstopandnofind}  \;
				}
			}	\label{paralgo:keystep2}
			
			Decompose $ x^h $ into $ \alpha^h $ and $ \gamma^h $ as in \eqref{alphagammadef}\;
			\If {$ a_{v^h} > b^h  - \omega^h $}
			{
				Update  $ \bar{n}\leftarrow v^h-1 $ and the partition $ \Pi $ as
				\begin{equation}
				\label{newpartition}
				\Pi \leftarrow \Pi \backslash \big\{\{0\ldots, \bar{n}\} \big \} \cup \big \{\{0,\ldots,v^h-1\}, \{v^h,\ldots,\bar{n}\} \big\};
				\end{equation}
			}
			Let $b^{h+1} = b^h-\omega^h $, $ x^{h+1} = \gamma^h $, and $ h\leftarrow h+1 $ \label{paralgo:recompstep} \;
		}
	%	\KwOut{Return the violated partition inequality or report that no one exists}
	\end{algorithm}
	\DecMargin{0.5em}

	\begin{mypro}%[\citet*{Pochet1995}]
		\label{intparmostviolated}
		Given a point $ \bar{x} \in Z_L(b) $ with $\bar{x} \notin \text{conv}(Z(b)) $, Algorithm {\rm \ref{intparseparalgo}} finds the most violated inequality of \eqref{intpartitionineq}.
	\end{mypro}
	\begin{proof}
		Since $ \bar{x} \notin \text{conv}(Z(b)) $, Algorithm \ref{intparseparalgo} stops in
		the step \ref{stopandfind1}, \ref{stopandfind2}, or \ref{stopandfind3}.
		If Algorithm \ref{intparseparalgo} stops in the step \ref{stopandfind1} or \ref{stopandfind2},
		\citet{Pochet1995} have shown that a most violated inequality of \eqref{intpartitionineq} can be found (See Theorem 13 in \cite{Pochet1995}).
		We now consider the case that Algorithm \ref{intparseparalgo} stops in the step \ref{stopandfind3}.
		
		Suppose that Algorithm \ref{intparseparalgo} stops at the $k$-th iteration and shows that $ \bar{x} $ violates some
		$ \Pi $-partition inequality by $ \epsilon $ but $ \bar{x} $ violates another $ \Pi' $-partition inequality by $ \epsilon' > \epsilon $. Define a new point $ \hat{x} $ as $ \hat{x}_0= \bar{x}_0 +\epsilon $ and $  \hat{x}_i= \bar{x}_i  $ otherwise. Since $ v^h \geq 1 $ during the first $ k-1 $ iterations, applying Algorithm \ref{intparseparalgo} to $ \hat{x} $, it is not difficult to verify that Algorithm \ref{intparseparalgo} will stop in the step \ref{newstopandnofind}. So $ \hat{x} \in \text{conv}(Z(b)) $. However, for the $ \Pi' $-partition inequality, as the coefficient of $ x_1 $ is $ 1 $, its violation at the point $ \hat{x} $ is $ \epsilon' - \epsilon > 0 $, which leads to a contradiction. This completes the proof. \qed
	\end{proof}
	
	\begin{myremark}
		\label{remark1}
		Suppose that Algorithm {\rm\ref{intparseparalgo}} terminates at the $k$-th iteration. With the steps {\rm\ref{paralgo:keystep1}-\ref{paralgo:keystep2}}, we know, the value $ x^h_0 $, $ h=1,\ldots, k $ are the same in Algorithm {\rm\ref{intparseparalgo}}. In addition, except the last iteration, we have $ v^h \geq 1 $ during the first $ k-1 $ iterations. This property plays an important role in the coming combinatorial separation algorithm for $ \text{conv}(X(b)) $.
	\end{myremark}
	
	The following complexity result is implied in \cite{Pochet1995} and stated explicitly in \cite{Hartmann1994}.
	\begin{theorem}%[\citet*{Pochet1995}]
		\label{complexityAlg1}
		The separation of the \rev{partition inequalities} \eqref{intpartitionineq} can be done in $ \mathcal{O}(n) $ time using Algorithm {\rm{\ref{intparseparalgo}}}.
	\end{theorem}

	\section{A combinatorial separation algorithm for the continuous $ \geq $-knapsack polyhedron}
	\label{section:separation}
	In this section, we study the separation problem of the $ \geq $-continuous knapsack polyhedron, i.e., $ \text{conv}(X(b)) $. Our purpose is to find an exact polynomial-time algorithm for separating the related family of partition inequalities, which is shown in \cite{Wolsey2016} to described $ \text{conv}(X(b)) $ with the trivial inequalities.
	
	We first introduce the partition inequalities for $ \text{conv}(X(b)) $.
	Let $ C \subseteq M $ such that $ b(C)= b - u(M\backslash C) > 0  $. Define $ \Pi $ as in \eqref{partition}
	be a partition of $ N_0$, where $ p \in \mathbb{Z}_{++} $ and $ a_{i_p} \leq b(C)  $. Let $ \beta_{p} =b(C)$  and define
	$ \beta_{t-1} $, $ \kappa_t $, $ \mu_t $ as in \eqref{paradef} for $ t=p,\ldots, 1 $.
	Then the partition inequality derived in \cite{Wolsey2016} is
	\begin{equation}
	\label{partitionineq}	
	s_0 +  s(C) +\sum_{i=1}^{j_1}\min\{a_i, \kappa_1\}x_i + \sum_{t=2}^p \prod_{l=1}^{t-1}\kappa_l \sum_{i=i_t}^{j_t} \min\bigg \{ \frac{a_i}{a_{i_t}}, \kappa_t \bigg \} x_i \geq \prod_{t=1}^{p}\kappa_t.
	\end{equation}
	As can be easily seen, different partitions of $ N_0 $ or subsets of $ M $ may lead to different partition inequalities of \eqref{partitionineq}. For convenience, we call the inequality \eqref{partitionineq} the $ (\Pi, C) $-partition inequality.
	Similarly, if the partition is $ \{ N_0 \} $, we call \eqref{partitionineq} the $ (N_0, C) $-partition inequality.% as the inequality
	
	The main difficulty in separating the family of partition inequalities \eqref{partitionineq} comes from its dependence on both the partition of $ N_0 $ and subsets of $ M $. On one hand, fixing the subset $ C \subseteq M $, the number of the inequalities \eqref{partitionineq} may be exponential due to the exponential number of partitions $ \Pi $. On the other
	hand, fixing the partition $ \Pi $ in \eqref{partition},
	the number of the inequalities \eqref{partitionineq} may also be exponential due to the exponential number of the subset $ C $. Therefore, given a point $ (\bar{x}, \bar{s}) $, a most violated partition inequality \eqref{partitionineq} would come from a very complicated procedure since both the subset $ C $ and the partition $ \Pi $ have a great impact on the generated partition inequality.

Nevertheless, the following observation indicates that given a fixed subset $ C \subseteq M$, it is easy to separate the partition inequalities  \eqref{partitionineq}.
	
	\begin{myobservation}
		\label{continteq}
		Given a point $　(\bar{x}, \bar{s}) \in X_L(b) $ and a fixed subset $ C \subseteq M $. Define a new point $ \hat{x} \in \mathbb{R}_+^{n+1}$ by setting $ \hat{x}_0 = \bar{s}_0+ \sum_{j \in C} \bar{s}_j$ and $ \hat{x}_i = \bar{x}_i $, $ i=1,\ldots,n $. Then, $ (\bar{x}, \bar{s}) $ violates the $ (\Pi,C) $-partition inequality \eqref{partitionineq} of $ X(b) $ if and only if $ \hat{x} $ violates the $ \Pi $-partition inequality \eqref{intpartitionineq} of $ Z(b(C))  $ and the violations are the same.
	\end{myobservation}
	Given a fixed subset $ C \subseteq M $, Observation \ref{continteq} indicates that, by a simple transformation, the partition inequalities \eqref{partitionineq} can be separated using Algorithm \ref{intparseparalgo}. For simplicity, we skip this simple transformation in below and use Algorithm \ref{intparseparalgo} to separate the partition inequality \eqref{partitionineq} for a fixed subset $ C $.
	%Therefore, given a fixed $ C_0 $, by replacing the term $ s(C_0) $ with $ s_0 $, we can use Algorithm 1 to separate continuous partition inequality \eqref{partitionineq}.
	
	Choosing the subset $ C \subseteq M $ is still nontrivial. In the remainder, however, we can show that by considering at most $ m+1 $ specific subsets of $ M $, a most violated partition inequality \eqref{partitionineq} can be found. Then it is sufficient to apply Algorithm \ref{intparseparalgo} to the partition \rev{inequalities \eqref{partitionineq}} restricted to these $ m+1 $ particular subsets of $ M $.
	%determine a priori which subsets of $ M $ can provide a most violated inequality.
	%
	%the subset $ C $ can be d
	%However, as we have mentioned, for any subset $ C $ and any partition $ \Pi $, no exact polynomial algorithm for separating the partition inequality \eqref{partitionineq} exists in the literature. In the next section, we shall resolve this problem by proposing a polynomial time algorithm.

	\subsection{\bf The selection of continuous variables in the $(N_0,C) $-partition inequality}
	We now consider the $ (N_0,C) $-partition inequality, which can be written as
	\begin{equation}
	\label{coefineq}
	s_0 +s(C)+  \sum_{i=1}^r a_i x_i + b(C) \sum_{i=r+1}^n x_i  \geq b(C),
	\end{equation}
	where $ r \in N_0 $ satisfying $ a_r \leq b(C)< a_{r+1} $. Notice that the derivation
	of $ (N_0,C) $-partition inequality \eqref{coefineq} does not depend on the divisible property of the coefficients $ a_i $, $ i=1,\ldots, n $. In general, \eqref{coefineq} is valid for $ X(b) $ with arbitrarily positive coefficients $ a_i $, $ i=1, \ldots, n $.
	
	Given a point $ (\bar{x}, \bar{s}) \in X_L(b) $, let $ j^+ $, $ j^- \in M $ such that $ \frac{\bar{s}_{j^+}}{u_{j^+}} \leq \frac{\bar{s}_{j^-}}{u_{j^-}}  $. We shall show that, in order to find a violated inequality of \eqref{coefineq}, the continuous variable $ s_{j^+} $ is more appropriate to incorporate into the inequality \eqref{coefineq} than the continuous variable $ s_{j^-} $.
	%More exactly, given a point $ (\bar{x}, \bar{s}) \in X_L(b) $, we either find an inequality of the form \eqref{coefineq} violated by this point or prove that no such one exists. Notice that as we have mentioned, due to the exponentially possible selection of the subset $ C $, there may exists exponential number of the inequality \eqref{coefineq}. Thought the number of the inequalities \eqref{coefineq} may be exponential, we will show that, by considering
	%at most $ n+1 $ inequalities, we can find a most violated inequality or prove that no such one exists.
	\begin{theorem}
		\label{cont}
		Given a point $ (\bar{x}, \bar{s}) \in X_L (b) $, let $ C \subseteq M  $ and $ j^+ $, $ j^- \in M $ be such that $ \frac{\bar{s}_{j^+}}{u_{j^+}} \leq \frac{\bar{s}_{j^-}}{u_{j^-}}  $, $ j^- \in C $, and $ j^+ \notin C $. If the $ (N_0, C) $-partition inequlaity \eqref{coefineq} is a violated inequality and its violation is $ \epsilon >0 $, then either the $ (N_0, C \cup \{j^+\}) $-partition inequality or the $ (N_0, C \backslash \{j^-\}) $-partition inequality is a violated inequality and its violation is at least $ \epsilon $.
	\end{theorem}
	\begin{proof}
		The statement is clearly true if the violation of the $ (N_0, C \backslash \{j^-\}) $-partition inequality is at least $ \epsilon $. Otherwise, let $ \xi $ be the difference of the left hand side and the right hand side of the $ (N_0,C)$-partition inequality, i.e.,
		\begin{equation}
		\xi =  \bar{s}_0+\bar{s}(C) + \sum\limits_{i=1}^{r} a_i \bar{x}_i + b(C)\sum\limits_{i=r+1}^{n} \bar{x}_i   - b(C) (= -\epsilon < 0),
		\end{equation}
		where $ r \in N_0 $ satisfying $ a_{r} \leq b(C) < a_{r+1} $.
		We are left with the following two cases:
		\begin{itemize}
			\item [(i)] $ b(C) \leq  u_{j^-} $;
			\item [(ii)] $ b(C) >  u_{j^-} $ and $ -\xi >  -\xi',   $ where
			 \begin{equation}
			 \xi' =  \bar{s}_0+\bar{s}(C\rev{\backslash \{j^-\}}) + \sum\limits_{i=1}^{r'} a_i \bar{x}_i + (b(C)-u_{j^-})\sum\limits_{i=r'+1}^{n} \bar{x}_i   - (b(C)-u_{j^-}).
			 \end{equation}
			 Here $ r' \in N_0 $ satisfies $ a_{r'} \leq b(C)-u_{j^-} < a_{r'+1}  $.
		\end{itemize}
		We shall complete the proof by showing that in both cases, the $ (N_0, C \cup \{j^+\}) $-partition inequality is violated by $ (\bar{x}, \bar{s}) $ by at least $ \epsilon $. We proceed with the following claim.
		\vspace{0.2cm}
		\begin{myclaim}
		\label{claim1}
		$ u_{j^-}\bigg (\sum\limits_{i=r+1}^n\bar{x}_i -1\bigg  ) + \bar{s}_{j^-} < 0$.
		\end{myclaim}
		{\emph{Proof of Claim \rm{\ref{claim1}}}}. We prove the claim by treating the two cases (i) and (ii) separately.
		Suppose that (i) holds. Since $ (\bar{x}, \bar{s}) \in X_L (b) $, it follows that $ \bar{x}_i \geq 0 $ for $ i \in N $ and $ \bar{s}_j \geq 0 $ for $ j \in M\cup \{0\} $. This, together with the fact that the inequality \eqref{coefineq} is violated by the point $ (\bar{x}, \bar{s}) \in X_L (b) $, implies that
		\begin{equation}
		\label{contineq1}
		\sum_{i=r+1}^n\bar{x}_i -1 < 0.
		\end{equation}
		In addition, we have that
		\begin{equation*}
		\begin{aligned}
		0 > \xi &  = & & \bar{s}_0+ \bar{s}(C) +  \sum\limits_{i=1}^{r} a_i \bar{x}_i + b(C)\sum\limits_{i=r+1}^{n} \bar{x}_i   - b(C) \\
		 & \geq  & & b(C)\sum_{i=r+1}^{n} \bar{x}_i - b(C) + \bar{s}_{j^-}.
		\end{aligned}
		\end{equation*}	
		%	This, combined with the fact that $ \xi < 0 $, indicates that
		%	\begin{equation*}
		%	b(C)\sum_{i=r+1}^{n} \bar{x}_i - b(C) + \bar{s}_{j^-} < 0.
		%	\end{equation*}
		Then it follows from \eqref{contineq1} and $  b(C) \leq u_{j^-} $ that
		\begin{equation*}
		\bar{s}_{j^-} < b(C) \bigg(1- \sum_{i=r+1}^{n} \bar{x}_i\bigg) \leq  u_{j^-}\bigg(1- \sum_{i=r+1}^{n} \bar{x}_i\bigg).
		\end{equation*}
		Hence the claim follows. Now consider the case (ii). In this case, noticing that from $ a_r \leq b(C) < a_{r+1} $ and $ a_{r'} \leq b(C)- u_{j^-} < a_{r'+1} $, we have $ r' \leq r $. Then
		\begin{equation}\label{claim22}
		0 >\xi - \xi'  = \sum_{i=r'+1}^{r}(a_i -
		b(C)+u_{j^-})\bar{x}_i + u_{j^-}\bigg (\sum_{i=r+1}^n\bar{x}_i -1 \bigg ) + \bar{s}_{j^-}.
		\end{equation}
		Since $ a_i - b(C) + u_{j^-} \geq   a_{r'+1}- b(C) + u_{j^-} > 0 $ and $ \bar{x}_i \geq 0 $ for $ i=r'+1, \ldots, r $, the relation (\ref{claim22}) implies the truth of the claim for case (ii). So the claim is always true.\vspace{0.2cm}\\
		Now, combining Claim \ref{claim1} and the assumption that $ \frac{\bar{s}_{j^+}}{u_{j^+}} \leq \frac{\bar{s}_{j^-}}{u_{j^-}} $, yields
		\begin{equation}
		\label{theoremineq1}
		\frac{\bar{s}_{j^+}}{u_{j^+}} \leq \frac{\bar{s}_{j^-}}{u_{j^-}} < 1 - \sum_{i=r+1}^n \bar{x}_i.				
		\end{equation}
		This is further equivalent to
		\begin{equation}
		\label{theoremineq2}
		u_{j^+}\bigg(\sum_{i=r+1}^n\bar{x}_i -1 \bigg) + \bar{s}_{j^+} < 0.
		\end{equation}
		Let $r'' \in N_0 $ such that $ a_{r''} \leq b(C) + u_{j^+} < a_{r''+1}$. Then $ r'' \geq r $. Similarly, denote
		\begin{equation*}
			\xi'' =  \bar{s}_0+\bar{s}(C\cup \{j^+\}) +   \sum\limits_{i=1}^{r''} a_i \bar{x}_i + (b(C) + u_{j^+})\sum\limits_{i=r''+1}^{n} \bar{x}_i - (b(C) + u_{j^+}).
		\end{equation*}
		Comparing $ \xi $ and $ \xi'' $, we have
		\begin{equation}
		\label{contineq4}
		\xi'' - \xi  = \sum_{i=r+1}^{r''}(a_i -
		b(C) - u_{j^+})\bar{x}_i + u_{j^+}\bigg(\sum_{i=r+1}^n\bar{x}_i -1 \bigg) + \bar{s}_{j^+} < 0.
		\end{equation}
		where the last inequality comes from \eqref{theoremineq2} and the fact that $ a_i - b(C) - u_{j^+} \leq  a_{r''} - b(C) - u_{j^+}\leq 0 $ and $ \bar{x}_i\geq 0  $ for $ i=r+1, \ldots, r'' $.
		Thus, the $ (N_0, C \cup \{j^+\}) $-partition inequality gives the violation $ -\xi'' >-\xi= \epsilon $ and the statement is true.\qed
	\end{proof}
	
	\begin{myremark}
		In the proof of Theorem {\rm\ref{cont}}, we do not use the divisible property of the coefficients $ a_i $, $ i=1, \ldots, n $. Therefore, The result in Theorem {\rm\ref{cont}} can be extended to the set $ X(b) $ with arbitrary positive coefficients $ a_i $, $ i=1, \ldots, n $.
	\end{myremark}

	\subsection{\bf The selection of continuous variables in the partition inequality with any partition}
	
	If fixing  the specific partition $ \{N_0\} $, Theorem \ref{cont} shows that in order to find a violated partition inequality, comparing to the variable $ s_{j^-} $, it is better to keep the variable $ s_{j^+} $ in the inequality \eqref{coefineq}. Next we shall show that this result can further be generalized to the partition inequality \eqref{partitionineq} with any partition.
	Before proceeding it, we illustrate the procedure of Algorithm \ref{intparseparalgo} for a fixed subset $ C $. Given a point $ (\bar{x}, \bar{s}) \in X_L(b) $, fixing the subset as $ C $, using Algorithm \ref{intparseparalgo}, it will generate the partition $ \Pi $. Let $ (x^h, \alpha^h, \gamma^h, b^h, r^h, v^h, \omega^h ) $, $ h=1,\ldots, k $, be the values which have been constructed in Algorithm \ref{intparseparalgo} terminating at the $ k $-th iteration where $ x^1 = \bar{x} $, $ b^1 = b(C)$, $ r^h $ satisfies $ a_{r^h} \leq b^h < a_{r^{h}+1} $, $ \delta^h $, $ \kappa^h $, and $ \omega^h $ are defined by
	\begin{equation}
	\label{intparsepaineqproof}
	\delta^h = \sum_{i=r^h+1}^n x^h_i,  \  \kappa^h = \bigg \lfloor \frac{b^h}{a_{r^h}} \bigg \rfloor +1, \ \omega^h = (\kappa^h -1)a_{r^h},
	\end{equation}
	$ v^h $ satisfies
	\begin{equation}
	\label{vcondproof}
	\sum_{i=v^h+1}^{r^h} a_ix^h_i < \omega^h(1-\delta^h) \ \text{and} \  \sum_{i=v^h}^{r^h} a_ix^h_i \geq \omega^h(1-\delta^h),  	
	\end{equation}
 $ \alpha^h $ and $ \gamma^h $ are defined by
	\begin{equation}
	\label{alphagammadefproof}
	\alpha^h_i = \left\{
	\begin{array}{ll}
	0,  &   i<v^h; \\
	\frac{\omega^h (1-\delta^h)-\sum_{i=v^h+1}^{r^h}a_i x^h_i}{a_{v^h}}, & i=v^h;\\
	x^h_i, & i > v^h
	\end{array}
	\right. \text{and} \ \
	\gamma^h_i = \left\{
	\begin{array}{ll}
	x^h_i - \alpha^h_i,  &   i\leq r^h; \\
	x^h_i, & i > r^h $ $  $ $
	\end{array}
	\right.
	\end{equation}
	for $ h=1,\ldots,k $, and
	$ b^{h+1} = b^h - \omega^h $ and $ x^{h+1} = \gamma^h $ for $ h=1,\ldots,k-1 $. Notice that from Remark \ref{remark1}, here we can assume that $ \bar{x}^h_0 = \bar{s}_0 + \bar{s}(C) $ for $ h =1, \ldots, k $. By definition, the following formula can easily be verified.
	\begin{myobservation}
		\label{bCformula}
		$ b(C) = b^t +   \sum\limits_{h=1}^{t-1} (\kappa^h -1)a_{r^h} $  for any $t =1, \ldots, k$.
	\end{myobservation}	
	
	We now consider the difference of generated partitions by Algorithm  \ref{intparseparalgo} for different subsets of $ M $. Consider another subset $ C' \subseteq M $ with $ C' \neq C $. Assume that $ a_{r^0} = + \infty $. Let $ q \leq k $  be the maximal integer such that
	\begin{equation}
	\label{sameparcond}
	b(C') = \sum\limits_{h=1}^{q-1} (\kappa^h -1)a_{r^h} + \theta \ \text{ and } \  0 \leq \theta < a_{r^{q-1}}.
	\end{equation}
	Define a new partition of $ N_0 $ as
	\begin{equation}
	\label{partition2}
	\begin{array}{ll}
	\Pi' = \big \{ & \{i_1=0,i_{1} +1 ,\ldots,j_{\tau} \}, \\
	&\{i_{\tau+1}=j_{\tau}+1,i_{\tau+1}+1,\ldots,j_{\tau+1}\}, \\
	&\qquad \vdots\\
	&\{i_p = j_{p-1}+1,i_p+1,\ldots,j_p = n \} \ \big\},
	\end{array}
	\end{equation}
	where \rev{$ \tau = \max \{ g:i_g < r^{q-1}, g =1, \ldots, p \} $}. Notice that by \eqref{newpartition}, in Algorithm \ref{intparseparalgo}, the partition may change as $ h $ grows. For notation convenience, we assume that at the $ 0 $-th iteration, the partition is $ \{ \{0,1,\ldots, n \}\}.$
	
	\begin{mylemma}
		\label{lemmasamepar}
     Assume that Algorithm {\rm{\ref{intparseparalgo}}} generates the partition $\Pi$ after $k$ iterations when fixing the subset of $ M $ as $C$.
   Consider another subset $ C' \subseteq M $ with $ C' \neq C $ and let $q\leq k$ be the maximal integer such that \eqref{sameparcond} holds. Then, if the subset of $ M $ is fixed as $C'$, Algorithm {\rm{\ref{intparseparalgo}}} yields the partition $ \Pi' $ defined in \eqref{partition2} at the $ (q-1) $-th iteration. Specially, if $ q = k $, $ \Pi' $ and $ \Pi $ are the same.
	\end{mylemma}
	\begin{proof}
		From Remark \ref{remark1}, we have $ v^h \geq 1 $ for $ h =1, \ldots,k-1 $. While applying Algorithm {\rm{\ref{intparseparalgo}}} with the subset $C'$, although the value of $ \bar{x}_0 $ is changed to $ \bar{s}_0 + \bar{s}(C') $, the case \eqref{vcondproof} is still satisfied for $ h=1, \ldots, q-1 $. Furthermore, by \eqref{sameparcond}, we know that except $ x^h_0 $, $ \alpha^h_0 $, and $ \gamma^h_0 $, the same values $ (x^h, \alpha^h, \gamma^h, r^h, v^h, \omega^h ) $ will be obtained for $ h=1, \ldots, q-1 $. This indicates that the partition $ \Pi' $ will be generated
at the $ (q-1) $-th iteration if applying Algorithm {\rm{\ref{intparseparalgo}}} with the subset as $C'$.
		Therefore, the statement follows.
		\qed
	\end{proof}
	
%	For the subset $ C $, suppose Algorithm {\rm{\ref{intparseparalgo}}} returns the $ (\Pi, C)  $-partition inequality. By Lemma \ref{lemmasamepar}, applying Algorithm {\rm{\ref{intparseparalgo}}} with the subset $ C' $, at the $ (q-1) $-th iteration, we will obtain the $ (\Pi', C') $-partition inequality.

Lemma \ref{lemmasamepar} reveals some invariance property of the partitions generalized by Algorithm
 {\rm{\ref{intparseparalgo}}} with different subsets of $M$. Combining this lemma with Theorem \ref{nontrivialresult}, we have the following result.
	%Next, we consider the question that under which condition, the subset $ C' $ $( C'\neq C) $ can give a partition inequality \eqref{partitionineq} whose violations is at least the same as that of $ (\Pi,C) $-partition inequality.
	\begin{mypro}
			\label{twopartition}
			Suppose the violation of $ (\Pi, C) $-partition inequality is $ \epsilon $.
			Let $ C' \subseteq M $ and $ q \leq k  $ be the maximal integer such that \eqref{sameparcond} holds.
			For $ X(b^q + u(M \backslash C)) $, if the $ (N_0, C') $-partition inequality is violated by the point $ (x^q, \bar{s}) $ by $ \epsilon' \geq \epsilon $, then for $ X(b) $, the violation of $ (\Pi', C') $-partition \rev{inequality} is also $ \epsilon' $ \rev{where $ \Pi' $ is defined in \eqref{partition2}}.
	\end{mypro}
	
	Proposition \ref{twopartition} gives the condition under which the subset $ C' $ $( C'\neq C) $ can give a partition inequality \eqref{partitionineq} whose violation is at least the same as the one by the subset $ C $. We are now ready to present the main result of this subsection, i.e., generalize Theorem \ref{cont} to the partition inequality \eqref{partitionineq} with any partition.
	\begin{theorem}
		\label{finalresult}
		Given a point $ (\bar{x}, \bar{s}) \in X_L (b) $,  let $ C \subseteq M  $ and $ j^+ $, $ j^- \in M $ such that $ \frac{\bar{s}_{j^+}}{u_{j^+}} \leq \frac{\bar{s}_{j^-}}{u_{j^-}}  $, $ j^- \in C $, and $ j^+ \notin C $. Suppose that there exists a partition inequality \eqref{partitionineq} given by $ C $ violated by the point $ (\bar{x}, \bar{s}) $ by $ \epsilon >0 $. Then there exists another violated partition inequality \eqref{partitionineq} given by $ C \cup \{j^+\} $ or $ C \backslash \{j^-\} $ and the violation is at least $ \epsilon $.
	\end{theorem}
	\begin{proof}
		%Let $ \epsilon $ be the violation of corresponding inequality of \eqref{partitionineq} given by $ C $.
	%	Let $ (\hat{x}_0, \hat{x}) \in \mathbb{R}_+\times \mathbb{R}_+^n $ such that
	%	\begin{equation}
	%		\hat{x}_0 = \bar{s}_0 + \bar{s}(C) \ \text{and} \ \hat{x}_i = \bar{x}_i \ \text{for}\ i=1, \ldots, n.
	%	\end{equation}
	Fixing the subset as $ C $, suppose Algorithm \ref{intparseparalgo} terminates with a partition $\Pi$ and the corresponding $ (\Pi,C) $-partition inequality.  From Proposition \ref{intparmostviolated}, its violation at the point $ (\bar{x}, \bar{s}) $ is at least $ \epsilon $. Without loss of generality, we can assume that, at the point $ (\bar{x}, \bar{s})  $, the violation of $ the
(\Pi, C) $-partition inequality is exactly $ \epsilon $. Let $ (x^h, \alpha^h, \gamma^h, b^h, r^h, v^h, \omega^h ) $, $ h=1,\ldots, k $, be the values which have been constructed in Algorithm \ref{intparseparalgo} terminating with $  (x^k, \bar{s}) \notin X(b^k+u(M\backslash C)) $.
	%	 Then it follows from Remark \ref{remark1} that
	%	 \begin{equation}
	%	 	\label{vhgeq1}
	%	 	v^h \geq  1 \ \text{for} \  h=1, \ldots, t-1.
	%	 \end{equation}
		 From Theorem \ref{nontrivialresult} and Observation \ref{continteq}, the $ (N_0, C) $-partition inequality for $ X(b^k+u(M\backslash C)) $ is
		\begin{equation}
		\label{tineq}
		s_0 +s(C) + \sum_{i=1}^{r^k}a_i x_i + b^k \sum_{i=r^k+1}^n x_i \geq  b^k
		\end{equation}
		and its violation at $  (x^k, \bar{s}) $ is $ \epsilon $.
		%	be the violated inequality for $ X(b^t) $ which is generated by Algorithm \ref{intparseparalgo}.
		%	Then, from Theorem \ref{nontrivialresult}, $ \epsilon $ must be the violation of \eqref{tineq} at the point $ (x^t, \bar{s}) $.
		By Theorem \ref{cont}, either (i) the $ (N_0, C \backslash \{j^-\}) $-partition inequality or (ii) the $ (N_0, C\cup \{j^+\}) $-partition inequality is a violated inequality for $ X(b^k+u(M\backslash C)),$ where the violation is $ \epsilon' \geq \epsilon $ at the point $ (x^k, \bar{s}) $.
		\rev{By Theorem \ref{cont}, the statement is true if $ k =1 $.} Hence, in what follows, we assume that $ k \geq 2 $. Then from the fact that $ b^k = b^{k-1} - \omega^{k-1} $ and the definition of $ \omega^{k-1} $ in \eqref{intparsepaineqproof}, we have $ b^{k} <a_{r^{k-1}} $.
		
		We first consider the case that (i). Clearly, $ 0 <  b^k  - u_{j^-} < a_{r^{k-1}}    $. By Observation \ref{bCformula}, we have
		\begin{equation}
		b(C\backslash\{j^-\}) = b(C) - u_{j^-} =  \sum\limits_{h=1}^{k-1} (\kappa^h -1)a_{r^h} + b^k -u_{j^-}.
		\end{equation}
		Then it follows from Proposition \ref{twopartition}
		that the $ (\Pi, C\backslash\{j^-\}) $-partition inequality is violated by the point $ (\bar{x}, \bar{s}) $ by $ \epsilon' \geq \epsilon $.

		Now we consider the case (ii).
		Again from Observation \ref{bCformula}, we have
		\begin{equation}
		\label{prooff1}
		b(C\cup\{j^+\}) = b(C) + u_j^+ = \sum\limits_{h=1}^{k-1} (\kappa^h -1)a_{r^h} + b^k +u_{j^+}.
	\end{equation}
		If $ b^k + u_{j^+} < a_{r^{k-1}}$,
		similar to the case (i), the $ (\Pi, C\cup \{j^+\} ) $-partition inequality \eqref{partitionineq} for $ X(b) $ is violated by $ (\bar{x}, \bar{s}) $ by $ \epsilon' \geq \epsilon $ and hence the statement follows. Otherwise, we have $ b^k + u_{j^+} \geq a_{r^{k-1}}$.
%Let $ q \leq k-1$ be the integer such that	
%		\begin{equation}
%			\label{arq-2}
%			  b^q + u_{j^+} < a_{r^{q-1}}, \text{\quad but\quad } b^{q+1} + u_{j^+} \geq a_{r^{q}}.
%		\end{equation}
Let $ q \leq k-1$ be the maximal integer such that	
		\begin{equation}
			\label{arq-1}
			  b^q + u_{j^+} < a_{r^{q-1}}.
		\end{equation}
		This implies that
		\begin{equation}\label{arq-2}
			b^{t+1} + u_{j^+} \geq a_{r^{t}} \ \ \text{for}\ t=q, q+1, \ldots, k-1.
		\end{equation}

		We shall complete the proof by showing that $ (x^{q}, \bar{s}) $ violates the $ (N_0,C \cup \{j^+\}) $-partition inequality \eqref{coefineq} for $ X(b^{q} + u(M\backslash C)) $ by at least $ \epsilon'$.
		From Proposition \ref{twopartition}, this implies that at the point $ (\bar{x}, \bar{s})$, the violation of the $ (\Pi', C\cup \{j^+\}) $-partition inequality \eqref{partitionineq} is also at least $ \epsilon' \geq \epsilon, $ where $ \Pi' $ is defined in \eqref{partition2}.	
		
		We proceed by induction on  the value of $ q $. Clearly, $ (x^k, \bar{s}) $ violates the $ (N_0, C \cup \{j^+\}) $-partition inequality \eqref{coefineq} for $ X(b^k + u(M\backslash C)) $ by $ \epsilon' $. Now suppose that the claim holds for $ k, k-1,\ldots, q+1 $.
		%Then $ (x^q, \bar{s}) $ violates the inequality \eqref{coefineq} defined by $ C_0 \cup \{j^+\} $ for $ X(b^q+u_j^+) $ by at least $ \epsilon $.
		Let
		\begin{equation}
		\label{pthineq}
		s_0 + s(C\cup \{j^+\}) + \sum_{i=1}^{r'}a_i x_i + (b^{q+1} + u_{j^+}) \sum_{i=r'+1}^n x_i\geq  b^{q+1} + u_{j^+}
		\end{equation}
		be the $ (N_0, C \cup \{j^+\}) $-partition inequality \eqref{coefineq} for $ X(b^{q+1} + u(M\backslash C)), $
		%be the violated inequality of \eqref{coefineq}
		where $ r' \in N_0$ satisfying $ a_{r'} \leq b^{q+1} + u_{j^+}  < a_{r'+1} $.
		Then it follows from \eqref{arq-2} that
		$ a_{r'+1} >b^{q+1} + u_{j^+} \geq a_{r^{q}} $. Hence $ r' \geq r^q $.
	%	\begin{equation}
	%	\label{rrq}
	%	r' \geq r^{q}. 	
	%	\end{equation}
		%	$ r' \geq r^{q-1} $.
		Let $ \xi $ be the difference of the left and right hand side of \eqref{pthineq} at the point $ (x^{q+1}, \bar{s}) $, i.e.,
		\begin{equation}
		\label{xi1}
		\xi =  \bar{s}_0 + \bar{s}(C\cup \{j^+\})+ \sum_{i=1}^{r'}a_i x^{q+1}_i + (b^{q+1} + u_{j^+}) \sum_{i=r'+1}^n x^{q+1}_i - ( b^{q+1} + u_{j^+}).
		\end{equation}
		By the assumption, $ -\xi \geq \epsilon' $.
		From Observation \ref{omegaequality}, we have that
		\begin{equation}
		\label{finaleq1}
		\sum_{i=v^{q}}^{r^{q}} a_i \alpha^{q}_i + \omega^{q} \sum_{i=r^{q}+1}^n \alpha_i^{q} = \omega^{q}.
		\end{equation}
		By \eqref{alphagammadefproof} and the fact that $ x^{q+1} = \gamma^{q} $, it is not difficult to verify that
		\begin{eqnarray}
			\alpha_i^{q} = 0 , \qquad \qquad \qquad \qquad  ~ & \qquad \ \ \ \ \ \,  \text{for} \ i=1, \ldots, v^{q}-1, \nonumber\\
			x_i^{q}= \gamma_i^{q} + \alpha_i^{q} = x_i^{q+1} + \alpha_i^{q}, &\qquad \text{for} \ i=1, \ldots, r^{q}, \label{mediaresult} \\
			\alpha_i^{q} =  x_i^{q} = \gamma_{i}^{q} = x_i^{q+1}, \ \qquad&\qquad \ \ \ \ \  \text{for} \ i=r^{q}+1, \ldots, n \nonumber.
		\end{eqnarray}
	%	This, together with \eqref{rrq}, \eqref{xi1}, and \eqref{finaleq1}, yields
		Adding \eqref{xi1} and \eqref{finaleq1}, we have
		\begin{eqnarray*}
			\xi & =  &   \bar{s}_0 + \bar{s}(C\cup \{j^+\})+ \sum_{i=1}^{r'}a_i x^{q+1}_i + (b^{q+1} + u_{j^+}) \sum_{i=r'+1}^n x^{q+1}_i - ( b^{q+1} + u_{j^+})  \\
			& & +\bigg (\sum_{i=v^{q}}^{r^{q}} a_i \alpha^{q}_i + \omega^{q} \sum_{i=r^{q}+1}^n \alpha_i^{q} - \omega^{q}\bigg)\\
				& = &  \bar{s}_0+ \bar{s}(C\cup \{j^+\})+\sum\limits_{i=1}^{r^q}a_i x^{q+1}_i + \sum\limits_{i=r^{q+1}}^{r'}a_i x^{q+1}_i  + (b^{q+1} + u_{j^+}) \sum\limits_{i=r'+1}^n x^{q+1}_i  \\
			& & + \sum_{i=v^{q}}^{r^{q}} a_i \alpha^{q}_i + \omega^{q} \sum_{i=r^{q}+1}^n \alpha_i^{q} - ( b^{q+1} + u_{j^+} + \omega^{q})  \qquad (\text{from} \ r' \geq r^q) \\
			& = &  \bar{s}_0+ \bar{s}(C\cup \{j^+\})+\sum\limits_{i=1}^{r'}a_i x^{q}_i + (b^{q+1} + u_{j^+}) \sum\limits_{i=r'+1}^n x^{q}_i + \omega^{q} \sum\limits_{i=r^{q}+1}^n x_i^{q} \\
			& &  - ( b^{q+1} + u_{j^+} + \omega^{q}). \qquad (\text{from} \ \eqref{mediaresult})
		\end{eqnarray*}
		On the other hand, let
		\begin{equation}
		\label{pminus1thineq}
		s_0 + s(C \cup \{ j^+ \})+\sum_{i=1}^{r''} a_i x_i + (b^{q+1} + u_{j^+} + \omega^{q}) \sum_{i=r''+1}^n x_i \geq  b^{q+1} + u_{j^+} + \omega^{q}
		\end{equation}
		be the $ (N_0,C\cup \{j^+\}) $-partition inequality for $ X(b^{q} + u(M\backslash C)) $ or equivalently,  $ X(b^{q+1} + \omega^{q} + u(M\backslash C)) $,
		where $ r'' \in N_0 $ satisfying $ a_{r''} \leq b^{q+1} + u_{j^+} + \omega^{q} < a_{r''+1} $.  Similarly, denote
		\begin{equation*}
		%	\begin{aligned}
		\xi' =  \bar{s}_0+\bar{s}(C \cup \{ j^+ \}) + \sum_{i=1}^{r''} a_i x^{q}_i + (b^{q+1} + u_{j^+} + \omega^{q}) \sum_{i=r''+1}^n x^{q}_i -  (b^{q+1} + u_{j^+} + \omega^{q}).
		%	\end{aligned}
		\end{equation*}
		Notice that $ r'' \geq r' \geq r^{q} $. Comparing $ \xi $ and $ \xi' $, we have
		\begin{equation}
		\begin{aligned}
		\xi' - \xi  & = & & \sum_{i=r'+1}^{r''} a_i
		x_i^{q} - (b^{q+1} + u_{j^+})\sum_{i=r'+1}^{r''}x_i^{q} - \omega^{q}\sum_{i=r^{q}+1}^{r''}x_i^{q}\\
		& = & & \sum_{i=r'+1}^{r''} (a_i -  b^{q+1}- u_{j^+} - \omega^{q})
		x_i^{q} - \omega^{q}\sum_{i=r^{q}+1}^{r'}x_i^{q} \leq 0,
		\end{aligned}
		\end{equation}
		where the last inequality follows from the fact that $ a_i \leq a_{r''} \leq  b^{q+1}  +u_{j^+} + \omega^{q} $ for all $ i = r' +1, \ldots, r'' $ and $ x_i^{q} \geq 0 $ for all $ i=r^{q}+1, \ldots, r'' $. This implies that the violation of the inequality \eqref{pminus1thineq} is $ -\xi' \geq -\xi \geq \epsilon'$. We finish the proof. \qed
	\end{proof}
	
%	\subsection{\bf Separation algorithm for the partition inequalities}

 \subsection{\bf A combinatorial separation algorithm for $\text{conv}(X(b))$ }
	
  Now we shall provide a combinatorial separation algorithm for the continuous $ \geq $-knapsack polyhedron $\text{conv}(X(b))$.

  Assume without loss of generality that $ \frac{\bar{s}_{1}}{u_{1}} \leq \cdots \leq \frac{\bar{s}_{m}}{u_{m}}$, since
  otherwise we can reorder the variable $ s_j $ according to the sorting of $ \frac{\bar{s}_{j}}{u_{j}} $, $ j=1,\ldots,m,$
  with the complexity of $ \mathcal{O}(m\log m) $. Furthermore, define
	\begin{equation}
	\label{Tdef}
	T_0=\varnothing  \ \ \text{and}\ \ T_j = \{ 1,\ldots, j \}, \ j = 1, \ldots, m.
	\end{equation}
	%The following theorem indicates that in order to find a most violated partition inequality \eqref{partitionineq}, we can only
%	consider the subsets $ T_j $, $ j=0,\ldots,m $.		
We are able to prove the following important theorem.

	\begin{theorem}
		\label{Ttauresult}
		Given the point $ (\bar{x}, \bar{s})\in X_{L}(b)$, suppose that $ \frac{\bar{s}_{1}}{u_{1}} \leq \cdots \leq \frac{\bar{s}_{m}}{u_{m}}  $. If there exists some violated partition inequality \eqref{partitionineq}, then there exists a most violated partition inequality \eqref{partitionineq} given by $ T_\tau $ for some $ \tau\in \mathbb{Z}_+ $ with $ 0 \leq \tau \leq m $.
	\end{theorem}
	\begin{proof}
		Define the set
		\begin{equation*}
		\begin{aligned}
			 \mathcal{C} = \big \{ C \subseteq M \, : \, \text{there exists a }
			 	\text{most violated inequality \eqref{partitionineq} given by $ C $} \big  \}.
		\end{aligned}
		\end{equation*}
		By definition, it is obvious that $ |\mathcal{C}| \geq 1$. Now if $ T_0 = \varnothing \in \mathcal{C}$, the statement is clearly true. Otherwise, for each $ C \in \mathcal{C} $, denote $ j^C_{\text{max}} = \max\{ j \, : \, j \in C \} $. %Let $ \tau = \min\{ j^C_{\text{max}} \, : \,  C \in \mathcal{C}\} $
		Let
		\begin{equation}
		\label{taudef}
		\tau = \min\{ j^C_{\text{max}} \, : \,  C \in \mathcal{C}\}
		\end{equation}
		and $ C^* \in   \argmin\{j^C_{\text{max}} \, : \,   C \in \mathcal{C}\} $. We shall prove that $ T_\tau \in \mathcal{C} $.
		
		As a matter of fact, if $ C^* = T_{\tau} $, we see that the statement is true. Otherwise, there exists $ j' <\tau$ such that $ j' \notin C^*  $. From Theorem \ref{finalresult}, either (i) $ C^*\backslash\{\tau\} $ or (ii) $ C^*\cup \{j'\} $ can lead to a most violated partition inequality. However, from the definition of $ \tau $ in \eqref{taudef}, the case (i) is not possible. Then $ C^*\cup \{j'\} \in \mathcal{C} $. Again, if $ C^*\cup \{j'\} = T_\tau $, we know the truth of the statement. Otherwise, we can repeat this process and the statement
follows naturally. \qed
	\end{proof}
	
Given a point $ (\bar{x}, \bar{s}) \in X_L(b) $, Theorem \ref{Ttauresult} indicates that, by considering only $ m+1 $ subsets of $ M $, i.e., $ T_j $, $ j=0,\ldots, m $, we can find a most violated inequality of \eqref{partitionineq} or prove $ (\bar{x}, \bar{s}) \in \text{conv}(X(b)) $. Then we can design a combinatorial separation algorithm,  Algorithm \ref{partitionineqalgo}, for the continuous $ \geq $-knapsack polyhedron $\text{conv}(X(b))$. 	
	
	\IncMargin{0.5em}
	\begin{algorithm}[ht!]
		%	\SetNlSty{}{\bf S}{}
		%	\SetNlSkip{0.2em}
		%	\SetAlgoNlRelativeSize{-2}
		\caption{The separation algorithm for the partition inequality \eqref{partitionineq} for $ X(b) $}
		\label{partitionineqalgo}
		\KwIn{The set $ X(b) $ and the point $ (\bar{x}, \bar{s})  \in  X_L(b)$}
		Reorder the variable $ s_j $, $ j \in M $, such that $ \frac{\bar{s}_1}{u_1} \leq \cdots \leq \frac{\bar{s}_m}{u_m} $\;\label{reorderstep}
		Initialize $ j^* = -1 $, $\epsilon^* = 0$, $ \text{ST}[0] = \bar{s}_0 $ and $ \text{BT}[0] = b - u(M) $\;
		\For{ $ j=1,\ldots, m $ }
		{
			\label{firstforstartloop}
			$ \text{ST}[j] = \text{ST}[j-1] + \bar{s}_j $\;
			$ \text{BT}[j] =  \text{BT}[j-1] + u_j  $\;
		}
		\label{firstforendloop}
		\For{ $ j=0,1,\ldots, m $}
		{
			\If{  $\text{\rm BT}[j]  > 0 $ }
			{
				Construct the point $ \hat{x} \in \mathbb{R}^{n+1} $ by setting $ \hat{x}_0 = \text{ST}[j]  $ and $  \hat{x}_i = \bar{x}_i  $ for $ i =1,\ldots, n $\; \label{constructnewpoint}
				\rev{For $ Z(\text{BT}[j]) $, use} Algorithm \ref{intparseparalgo} to find a partition inequality \eqref{intparsepaineq} violated by  $ \hat{x} $\;\label{callAlg1}
				\If{ {\rm a violated partition inequality \eqref{intparsepaineq} is found} }
				{
					Let $ \epsilon$ and $ \Pi $ be the associated violation and the partition, respecitively\;
					\If{ $ \epsilon > \epsilon^* $ }
					{
						$ \Pi^* \leftarrow \Pi $, $ j^*  \leftarrow j $, and $ \epsilon^* \leftarrow \epsilon $\;
					}
				}
			}	
		}
		\eIf { $ \epsilon^* > 0 $ }
		{Construct the $ (\Pi^*, T_{j^*}) $-partition inequality \eqref{partitionineq}\;}
		{Report that no such inequality exists\;}
		%	\KwOut{Return the violated partition inequality or report that no one exists}
	\end{algorithm}
	\DecMargin{0.5em}
	In Step \ref{reorderstep} of Algorithm \ref{partitionineqalgo}, we require to reorder the variable $ s_j $, $ j\in M $, with the complexity of  $ \mathcal{O}(m \log m) $. Then we compute $  \text{ST}[j] = s(T_j \cup\{0\} )  $ and $  \text{BT}[j] = b(T_j)  $ from Step \ref{firstforstartloop} to \ref{firstforendloop} with the complexity of $ \mathcal{O}(m) $. From Theorem \ref{complexityAlg1} and Observation \ref{continteq}, for a fixed subset $ T_j $, the partition inequalities \eqref{partitionineq} can be separated in $ \mathcal{O}(n) $ using Algorithm \ref{intparseparalgo}. This can be done from Step \ref{constructnewpoint} to Step \ref{callAlg1}. In summary, we have the following computational complexity result.
	
	%From Observation \ref{continteq}, for a fixed $ T_i $, the partition inequality of \eqref{partitionineq} can be separated in $ \mathcal{O}(n) $. The complexity of sorting the $ \frac{\bar{s}_{i}}{u_{i}} $, $ i=1,\ldots,m $ is $ \mathcal{O}(m \log m) $. Computing $ s(T_i \cup \{0\}) $, $ i=1,\ldots, m $, can be done in $ \mathcal{O}(m) $.
	%In summary, we have the following.
	\begin{theorem}
		The problem of separating the partition inequalities \eqref{partitionineq} can be solved with
        the complexity of $ \mathcal{O}(mn + m\log m) $.
	\end{theorem}

	\section{A combinatorial separation algorithm for the continuous $ \leq $-knapsack polyhedron}
	\label{lessknapsackset}
	In this section, we are concerned with the continuous $ \leq $-knapsack set
	\begin{equation*}
	Y(b) = \bigg\{ (x, y) \in \mathbb{Z}^{n}_+ \times \mathbb{R}^{m+1}_+ \, : \,
	\sum_{i\in N} a_i x_i \leq b + y_0 + \sum_{j \in M}y_j, \ 0 \leq y_j \leq u_j, j\in M \bigg\}.
	\end{equation*}
	where $ 1 | a_1 | \cdots | a_n $.
	At first, we present the family of complemented partition inequalities which, together with the constraints in $ Y(b) $, is shown to describe $ \text{conv}(Y(b)) $. It is proved that the family of complemented partition inequalities can be separated in polynomial time. We also discuss its relation to the family of $ \leq $-partition inequalities in \citet{Wolsey2016}.
	%First, consider its knapsack constraint
	%\begin{equation}\label{originalknap}
	%\sum_{i \in N}a_i x_i \leq b + s_0 + \sum_{j \in M}s_j.
	%\end{equation}
	
	\subsection{\bf Complemented partition inequality and a separation algorithm for \text{conv}(Y(b))}
	Consider the set $ Y(b) $. By introducing a slack variable $ s_0 $ and complementing the variable $ y_j $ using $ y_j = u_j - s_j  $ for all $ j \in M $, we obtain an equivalent set
	\begin{equation*}
		\begin{aligned}
		\bigg \{ (x, y_0, s) \in \mathbb{Z}_+^n \times \mathbb{R}_+ \times \mathbb{R}_+^{m+1} \, : \,
		s_0 + \sum_{j \in M}s_j + \sum_{i \in N}a_i x_i = y_0+   b + \sum_{j \in M}u_j , & \\
		\ s_j \leq u_j, \ j \in M \bigg \}.
		\end{aligned}
	\end{equation*}
	%\begin{equation}\label{eqknap}
	%s'_0 + \sum_{i \in N}a_i x_i= b + s_0 + \sum_{j \in M}s_j.
	%\end{equation}
	After disregarding the nonnegative variable $ y_0 $, we obtain the relaxation continuous $ \geq $-knapsack set $ X(b+u(M)) $, i.e., 
	\begin{equation*}
	\begin{aligned}
	\rev{\bigg \{ (x, s) \in \mathbb{Z}_+^n \times  \mathbb{R}_+^{m+1} \, : \,
	s_0 + \sum_{j \in M}s_j + \sum_{i \in N}a_i x_i \geq   b + \sum_{j \in M}u_j ,
	\ s_j \leq u_j, \ j \in M \bigg \}}.
	\end{aligned}
	\end{equation*}
	The following easily verified property tells us the isomorphism of $ X(b+u(M)) $ and $ Y(b) $.
	\begin{myobservation}
		\label{isomorphism}
		Given $ (\bar{x},\bar{y}\,) \in \mathbb{R}^n_+ \times \mathbb{R}^{m+1}_+ $, define $ (\bar{x},\bar{s}) $ by setting
		\begin{equation}
		\label{pointxtide}
		\bar{s}_0 = b + \bar{y}_0 + \sum_{j \in M}\bar{y}_j -  \sum_{i \in N}a_i \bar{x}_i \ \ \text{and}  \ \ \bar{s}_j = u_j - \bar{y}_j  \  \ \text{for} \ \  j \in M  .
		\end{equation}
		Then $ (\bar{x},\bar{y}) \in \text{conv}(Y(b)) $ if and only if $ (\bar{x},\bar{s}) \in  \text{conv}(X(b+u(M)) ) $.
	\end{myobservation}
	%and complementing the variables using $ s'_j = u_j - s_j $ for all $ j \in M $, we obtain the $ \geq $-continuous knapsack set
	%\begin{equation*}
	%\begin{aligned}
	%X'(b+u(M) ) = \big \{ (x, s') \in \mathbb{Z}_+^n \times \mathbb{R}_+^{m+1} \, : \,  s'_0 + \sum_{j \in M}s'_j + \sum_{i \in N}a_i x_i \geq   b + \sum_{j \in M}u_j  , & \\
	% \ s'_j \leq u_j, \ j \in M  \big \}.
	%\end{aligned}
	%\end{equation*}
	We further show the equivalence of $ Y(b) $ and $ X(b+u(M)) $ in the sense that there exists a one-to-one correspondence between the facet defining inequalities for $ \text{conv}(Y(b)) $ and those for $ \text{conv}(X(b+u(M))) $.
	\begin{mypro}
		\label{complementproof1}
		Let $ C \subseteq M $. Then
		\begin{equation}\label{newfacetineqxtide}
		 s_0 +s(M\backslash C) + \sum_{i \in N}\alpha_i x_i  \geq \gamma
		\end{equation}
		is valid for $ X(b+u(M)) $ if and only if
		\begin{equation}\label{newfacetineqx}
			\sum_{i\in N} (a_i - \alpha_i)x_i \leq b + u(M\backslash C)- \gamma + y_0 + y(C)
		\end{equation}
		is valid for $ Y(b) $.
		Moreover, \eqref{newfacetineqxtide} is facet defining for $ \text{conv}(X(b+u(M))) $ if and only if \eqref{newfacetineqx} is facet defining for
		$ \text{conv}(Y(b)) $.
	\end{mypro}	
	\begin{proof}
		For each of the two parts, we only prove the necessity since the proof of sufficiency is similar. First, we consider the first part.
		Let $ (\bar{x}, \bar{y}) $ be any point of $ Y(b) $. Then
		we have
		\begin{equation}\label{pro2ineq1}
			\sum_{i \in N}a_i \bar{x}_i \leq b + \bar{y}_0 + \sum_{j \in M}\bar{y}_j.
		\end{equation}
		Construct a new point $ (\bar{x}, \bar{s}) \in \mathbb{Z}_+^n \times \mathbb{R}_+^{n+1},$ where $ \bar{s} $ is defined as in \eqref{pointxtide}.
		It follows immediately that $ \bar{s}_0 \geq 0 $ and $ 0 \leq \bar{s}_j \leq u_j $ for $ j \in M $. Combining it with the fact that
		\begin{equation*}
			\begin{aligned}
			\bar{s}_0 + \sum_{j \in M}\bar{s}_j + \sum_{i \in N}a_i \bar{x}_i & = & &\!\!\!\!  b + \bar{y}_0 + \sum_{j \in M}\bar{y}_j -  \sum_{i \in N}a_i \bar{x}_i + \sum_{j \in M}(u_j-\bar{y}_j ) +\sum_{i \in N}a_i \bar{x}_i    \\
			& = & &\!\!\! \!  b + \bar{y}_0+ \sum_{j \in M}u_j  \geq  b + \sum_{j \in M}u_j,
			\end{aligned}
		\end{equation*}
		we have  $ (\bar{x}, \bar{s})  \in X(b+u(M))  $. The validity of \eqref{newfacetineqxtide} indicates that
		\begin{equation}\label{validtildex}
		 \bar{s}_0 + \bar{s}(M\backslash C) +\sum_{i \in N}\alpha_i \bar{x}_i \geq \gamma.
		\end{equation}
		By substituting \eqref{pointxtide} in \eqref{validtildex}, it follows
		\begin{equation*}
		\sum_{i \in N}(a_i - \alpha_i) \bar{x}_i \leq b + u(M\backslash C)- \gamma  + \bar{y}_0 + \bar{y}(C).
		\end{equation*}
		So \eqref{newfacetineqx} is a valid inequality for $ Y (b) $. Furthermore, if $ (\bar{x}, \bar{s}) $ satisfies \eqref{newfacetineqxtide} at equality, $ (\bar{x}, \bar{y}) $ also satisfies \eqref{newfacetineqx} at equality. This implies that if \eqref{newfacetineqxtide} is facet defining for $ \text{conv}(Y(b)) $, then \eqref{newfacetineqx} is also facet defining for $ \text{conv}(X(b+u(M))) $, which completes the proof. \qed
	\end{proof}
	
	%\begin{mypro}
	%	\label{complementproof2}
	%	Let $ C \subseteq M $ and
	%	\begin{equation}\label{newfacetineqx2}
	%	\sum_{i \in N}\beta_i x_i \leq \lambda + s_0 + s(M\backslash C)
	%	\end{equation}
	%	be a valid inequality of $ X_{\leq}(b) $. Then
	%	\begin{equation}\label{newfacetineqxtide2}
	%	s'_0 + s'(C) + \sum_{i\in N} (a_i - \beta_i )x_i \geq  b - \lambda + u(C)
	%	\end{equation}
	%	is a valid inequality of $X'_{\geq}(b+u(M))$.
	%	Moreover, if the inequality \eqref{newfacetineqx2} is facet defining for $ \text{conv}(X_{\leq}(b)) $, the inequality \eqref{newfacetineqxtide2} is facet defining for $ \text{conv}(X'_{\geq}(b+u(M))) $.	
	%\end{mypro}	
	%\begin{proof}
	%	
	%\end{proof}
	%
	%
	%\begin{proof}
	%	Similar to the proof of Proposition \ref{complementproof1}.
	%\end{proof}

	The result of Proposition \ref{complementproof1} is stated in \citet{Atamturk2010} when $ M= \varnothing $. Based on Proposition \ref{complementproof1}, we now describe the complemented partition inequalities for $ Y(b) $. Let $ \Pi $ be the partition of $ N_0 $ as defined in \eqref{partition}, $  C \subseteq M $, and $ B(C) = b +u(M\backslash C) $. Start with $ \beta_p = B(C) $ and define
	$ \kappa_t, \mu_t, \beta_{t-1} $ for $ t=p,\ldots, 1 $ as in \eqref{paradef}. The $ (\Pi, M\backslash C) $-partition inequality for $ X(b + u(M)) $ is
	\begin{equation}
	\label{lesspartitionineq}	
	s_0 +  s(M\backslash C) +\sum_{i=1}^{j_1}\min\{a_i, \kappa_1\}x_i + \sum_{t=2}^p \prod_{l=1}^{t-1}\kappa_l \sum_{i=i_t}^{j_t} \min\bigg \{ \frac{a_i}{a_{i_t}}, \kappa_t \bigg \} x_i \geq \prod_{t=1}^{p}\kappa_t.
	\end{equation}
	Using Proposition \ref{complementproof1}, we obtain the valid inequality
	\begin{equation}\label{comppartitionineq}
	\sum_{i\in N} (a_i - \alpha_i)x_i \leq B(C)- \prod_{t=1}^{p}\kappa_t + y_0 + y(C),
	\end{equation}
	 for $ Y(b) $ where $ \alpha_i $ is the coefficient of the variable $ x_i $ in the inequality \eqref{lesspartitionineq}. We call \eqref{comppartitionineq} the {\it complemented partition inequality}. \citet{Wolsey2016} have shown that with the addition of the family of the inequalities \eqref{lesspartitionineq}, the convex hull of $ X(b+u(M)) $ is completely described. This, together with Proposition \ref{complementproof1}, implies that the following result.
	\begin{theorem}
		\label{convhullY}
		The convex hull of $ Y(b) $ is described by its constraints and the complemented partition inequalities \eqref{comppartitionineq}.
	\end{theorem}
	
	From Theorem \ref{convhullY}, the separation problem of $ \text{conv}(Y(b)) $ can be reduced to,
	given a point $ (\bar{x},\bar{y}) \in Y_L(b) $, either finding a violated complemented partition inequality \eqref{comppartitionineq} or proving that $  (\bar{x},\bar{y}) \in  \text{conv}(Y(b)) $. By further combining with Observation \ref{isomorphism}, we know, the problem of finding
	a complemented partition inequality \eqref{comppartitionineq} violated by $ (\bar{x}, \bar{y}\,) $ is equivalent to that of finding an inequality of \eqref{lesspartitionineq} violated by $ (\bar{x}, \bar{s}) $ where $ \bar{s} $ is defined as in \eqref{pointxtide}.
This leads to the following separation algorithm.

\IncMargin{0.5em}
 \begin{algorithm}[ht!]
  \caption{The separation algorithm for the complemented partition inequality \eqref{comppartitionineq} for $ Y(b) $}
  \label{cpartitionineqalgo}
  \KwIn{The set $ Y(b) $ and the point $ (\bar{x}, \bar{y})  \in  Y_L(b)$}
  Construct the point $ (\bar{x}, \bar{s}) $ by \eqref{pointxtide}\;
  For the point $ (\bar{x}, \bar{s}) $, call Algorithm \ref{partitionineqalgo} to test whether a violated partition inequality \eqref{lesspartitionineq} exists for $ X(b+u(M)) $ \;
  \eIf {A violated partition inequality \rev{\eqref{lesspartitionineq}} is found}
  {Construct the complemented partition inequality \eqref{comppartitionineq}\;}
  {Report that no such inequality exists\;}
  % \KwOut{Return the violated partition inequality or report that no one exists}
 \end{algorithm}
 \DecMargin{0.5em}

As Algorithm \ref{partitionineqalgo} can be done in $ \mathcal{O}(mn+m\log m) $ time using Algorithm \ref{partitionineqalgo}, it follows immediately that the separation problem of $ \text{conv}(Y(b)) $ can be solved in the same complexity.

	\begin{theorem}
		The separation problem of $ \text{conv}(Y(b)) $ can be solved with the complexity of $ \mathcal{O}(mn+m\log m) $.
	\end{theorem}
	
	\subsection{\bf Connection with the $ \leq $-partition inequalities}
	%We close this section by noting that the complemented $ (\Pi, C) $-partition inequality is strong than the
	%$ \leq  $-partition inequality proposed by \citet{Wolsey2016}.
	We now discuss the relation between the complemented partition inequality \eqref{comppartitionineq} and the $ \leq  $-partition inequality presented in \cite{Wolsey2016}. We first introduce the $ \leq  $-partition inequality. Let $ g $ be the smallest index $ i $ such that $ a_i $ does not divide $ B(C) $ and $ q \in \{g, \ldots, n \} $.
	%Let $ \{i_2 =q, \ldots, j_2\}, \cdots , \{i_p, \ldots, j_p = n \} $ be a partition of $ \{q,\ldots, n\} $. Define $ s_t = c_{i_t}  - \beta_{t-1}$ for $ t=p,\ldots, 1  $.
	We use the same partition of $ \Pi $ in \eqref{partition} with the restriction that $ i_2 = q $. Notice that  $ \big \{\{i_2 =q, \ldots, j_2\}, \cdots , \{i_p, \ldots, j_p = n \} \big \}$ is  a partition of $ \{q,\ldots, n\} $. Start with $ \beta_p = B(C) $ and define
	$ \kappa_t, \mu_t, \beta_{t-1} $ for $ t=p,\ldots, 1 $ as in \eqref{paradef}. Let $ \lambda_t = a_{i_t} - \beta_{t-1} $ for $ t =p, \ldots, 2 $. Define $ \pi_2 = \lambda_2 $, $ \pi_t = \kappa_{t-1}\pi_{t-1} + (\lambda_t - \lambda_{t-1}) $ for $ t = 3, \ldots, p $ and $ \pi_0 = \kappa_p \pi_p - \lambda_p $. The $ \leq $-partition inequality in \citet{Wolsey2016} is
	\begin{equation}\label{wolseypartitionineq}
		\sum_{t=2}^p \pi_t \sum_{i =i_t}^{j_t} \frac{a_i}{a_{i_t}} x_i \leq \pi_0+ y_0 +y(C).
	\end{equation}
	Next, we give a closed form of \rev{$ \pi_t $ for $ t=0, 2, 3, \ldots, p$}.
	\begin{mylemma}
		\label{pilemma}
		$ \pi_t =a_{i_t} - \prod\limits_{l=1}^{t-1} \kappa_l$ for $ t=2,\ldots, p $ and $ \pi_0 = B(C)- \prod\limits_{t=1}^{p} \kappa_t$.
	\end{mylemma}
	\begin{proof}
		First, we evaluate
		\begin{equation}
		\label{sformu}
		\lambda_{t} = a_{i_{t}} - \beta_{t-1} =  a_{i_{t}} - [\beta_{t} - (\kappa_{t}-1)a_{i_{t}}] = \kappa_{t} a_{i_{t}}  - \beta_{t}
		\end{equation}
		for $ t=2,\ldots, p $.
		We prove $ \pi_t =a_{i_t} - \prod\limits_{l=1}^{t-1} \kappa_l$ for $ t =2, \ldots, p $ by induction. Since $ i_1 = 0 $ in \eqref{partition}, it is easy to see that $ \kappa_1 = \beta_1 $. Hence $ \pi_2 = \lambda_2 = a_{i_2} - \beta_1 = a_{i_2} - \kappa_1$,
which shows the truth of the statement for $t=2$. Assume that the statement is true for some $2\le t\le p-1$. Then
		\begin{equation}
		\begin{aligned}
			\pi_{t+1}  &= &&  \kappa_{t} \pi_{t} + (\lambda_{t+1} - \lambda_{t}) \\
						& = &&\kappa_{t}( a_{i_{t}}- \prod\limits_{l=1}^{t-1}\kappa_l) +  (\lambda_{t+1} - \lambda_{t})\\
						& = & & (\kappa_{t} a_{i_{t}}+  \lambda_{t+1} - \lambda_{t}) - \prod\limits_{l=1}^{t}\kappa_l\\
						& = & & (\kappa_{t} a_{i_{t}} + a_{i_{t+1}} - \beta_{t} - \lambda_{t} ) - \prod\limits_{l=1}^{t}\kappa_l\\
						& = & & a_{i_{t+1}} - \prod\limits_{l=1}^{t}\kappa_l,
		\end{aligned}
		\end{equation}	
		where the last equality comes from \eqref{sformu}. So the statement is also true for $t+1$. By induction, \rev{we} know that the statement
        holds for $ t =2, \ldots, p. $
		
		Furthermore, combining \eqref{sformu} and the fact that $ \beta_p = B(C) $ yields $ \lambda_{p} = \kappa_p a_{i_p} - B(C) $. This, together with $ \pi_p =  a_{i_p} - \prod\limits_{t=1}^{p-1}\kappa_t $, indicates
			\begin{equation*}
		\begin{aligned}
		\pi_0  = \kappa_p \pi_p - \lambda_p  = \kappa_p \bigg(a_{i_p} - \prod\limits_{t=1}^{p-1}\kappa_t\bigg ) - (\kappa_p a_{i_p} - B(C))    = B(C) -  \prod\limits_{t=1}^{p}\kappa_t.
		\end{aligned}	
		\end{equation*}
 This completes our proof. \qed
 \end{proof}

	We now show that the complemented partition inequality \eqref{comppartitionineq} is stronger than the $ \leq $-partition inequality \eqref{wolseypartitionineq}.
	\begin{mypro}
		\label{finnalpro}
		For the same partition $ \Pi $ of $ N_0 $ and the same subset $ C \subseteq M $, the complemented partition inequality \eqref{comppartitionineq} is stronger than the $ \leq $-partition inequality \eqref{wolseypartitionineq}. Moreover, the two inequalities are equivalent if and only if
		\begin{equation}
			\label{proeq1}
			\min\bigg\{\frac{a_i}{a_{i_t}}, \kappa_t \bigg\} = \frac{a_i}{a_{i_t}} \ \ \text{for} \ \ i=i_t, \ldots,j_t \ \ \text{with}  \ t=1, \ldots, p.
		\end{equation}
	\end{mypro}
	\begin{proof}
		%For the first part.
		By Lemma \ref{pilemma}, the right hand sides of both inequalities \eqref{comppartitionineq} and \eqref{wolseypartitionineq} are the same. We prove the statement by showing that the coefficient of each variable $ x_i $ in \eqref{wolseypartitionineq}, denoted by $ \sigma_i $, is less than or equal to that in \eqref{comppartitionineq}. For each $ i=1,\ldots, j_1$, we have $ \sigma_i = 0 $. In \eqref{comppartitionineq}, the coefficient of $ x_i $ is
		\begin{equation}
			\label{eq1ineq1}
			a_i - \alpha_i = a_i - \min\{a_i, \kappa_1\}  \geq 0= \sigma_i.
		\end{equation}
		For each $ i_t \leq i\leq {j_t} $ with $ t \geq 2 $, as $ \alpha_i $ is the coefficient of $ x_i $ in \eqref{lesspartitionineq}, it follows
		\begin{equation}
				\label{eq1ineq2}
		\begin{aligned}
			a_i - \alpha_i & = & & a_i - \prod_{l=1}^{t-1}\kappa_l  \min\bigg \{ \frac{a_i}{a_{i_t}}, \kappa_t \bigg \} \geq  a_i - \prod_{l=1}^{t-1}\kappa_l  \frac{a_i}{a_{i_t}}\\
				& = & & \bigg (a_{i_t} - \prod_{l=1}^{t-1}\kappa_l \bigg )\frac{a_i}{a_{i_t}}  = \pi_t \frac{a_i}{a_{i_t}} = \sigma_i.	\qquad \qquad \qquad \qquad
		\end{aligned}
		\end{equation}
		Here notice that $ a_{i_t} - \prod_{l=1}^{t-1}\kappa_l = \pi_t  $ is used, which is from Lemma \ref{pilemma}.
		Finally, \eqref{eq1ineq1} and \eqref{eq1ineq2} hold at equality if and only if \eqref{proeq1} holds, which completes the second part. \qed
	\end{proof}
	
  Even with the same partition $ \Pi $ of $ N_0 $ and the same subset $ C \subseteq M $, the complemented partition inequality \eqref{comppartitionineq} is stronger than $ \leq $-partition inequality \eqref{wolseypartitionineq}. In general, there always exists some $ \leq $-partition inequality defined by some $ \Pi' $ and $ C $ such that this inequality is the same as \eqref{comppartitionineq}. To see this, the following observation is important.
	
	\begin{myobservation}
		\label{findobservation}
		For some fixed $ \tau \in \mathbb{Z}$ with $ 1\leq \tau \leq p $, if $ \frac{a_{j_{\tau}}}{a_{i_{\tau}}} > \kappa_{\tau}  $, modifying the partition $ \Pi $ by changing the $ \tau $-th block into two blocks
		$$
			\{i_\tau, \ldots, j_{\tau-1}\} \ \text{and} \ \{j_\tau\},
		$$
		the associated partition inequality
		\rev{\begin{equation*}
			\begin{aligned}
			s_0 +  s(M\backslash C) +\sum_{i=1}^{j_1}\min\{a_i, \kappa_1\}x_i + \sum_{t=2}^{\tau-1} \prod_{l=1}^{t-1}\kappa_l \sum_{i=i_t}^{j_t} \min\bigg \{ \frac{a_i}{a_{i_t}}, \kappa_t \bigg \} x_i    & \\
			 + \prod_{l=1}^{\tau-1}\kappa_l \sum_{i=i_\tau}^{j_\tau-1}\min\bigg \{ \frac{a_i}{a_{i_\tau}}, \kappa_\tau \bigg \} x_i+  \prod_{l=1}^{\tau}\kappa_l x_{j_\tau} 
			  & \\ 
			  + \sum_{t=\tau+1}^{p} \prod_{l=1}^{t-1}\kappa_l \sum_{i=i_t}^{j_t} \min\bigg \{ \frac{a_i}{a_{i_t}}, \kappa_t \bigg \} x_i   &  \geq\prod_{t=1}^{p}\kappa_t .
			\end{aligned}
		\end{equation*}}
		is the same as \eqref{lesspartitionineq} {\rm(}and hence the same complemented partition inequality \eqref{comppartitionineq} is obtained{\rm)}.
	\end{myobservation}
	If fixing the subset $ C \subseteq M $, by applying Observation \ref{findobservation} to the complemented partition inequality \eqref{comppartitionineq} recursively, we will obtain a new partition $ \Pi' $. Meanwhile, for the corresponding complemented partition inequality, \eqref{proeq1} is satisfied. By Proposition \ref{finnalpro}, the $ \leq $-partition inequality \eqref{wolseypartitionineq} defined by $ \Pi' $ and $ C $ is equivalent to the complemented partition inequality \eqref{comppartitionineq} defined by $ \Pi $ and $ C $. \rev{Hence, the number of complemented partition inequalities \eqref{comppartitionineq} is smaller than that of the $ \leq $-partition inequalities \eqref{wolseypartitionineq}.}
	%Let $ B(C) = b + a(M\backepsilon C) $ and
	%\citet{Wolsey2016} show

	%From Theorem \eqref{finalresult}
	\section{Conclusions and remarks}
	\label{section:conclusions}
	In this paper, we have presented the separation algorithms for the continuous knapsack polyhedra with divisible capacities.
	In particular, for the continuous $ \geq $-knapsack polyhedron, we have given a polynomial-time combinatorial separation algorithm for the exponential family of $ \geq $-partition inequalities. We have shown that by considering the $ m+1 $ subsets of $ M $, namely, $ T_i $, $ i=0,\ldots, m $ as defined in \eqref{Tdef}, a most violated inequality of \eqref{partitionineq} can be found. This reduces the problem of separating the inequalities \eqref{partitionineq} to the problem of separating the inequalities \eqref{intpartitionineq}, which can be solved using Algorithm \ref{intparseparalgo}. For the continuous $ \leq $-knapsack polyhedron, we have derived the family of complemented partition inequalities by the complemented $ \geq $-knapsack set. Moreover, we have proved that, together with the initial constraints, a complete description of the continuous $ \leq $-knapsack polyhedron is obtained. This in turn solves the separation problem of the continuous $ \leq $-knapsack polyhedron.
	
	There still exist some other polyhedra whose polyhedral structure is known but the associated combinatorial separation problem remains open, see, e.g., the continuous knapsack polyhedron with two integer variables and arbitrary coefficients of integer variables \cite{Dash2016} and a mixing polyhedron variant \cite{DiSumma2011}. Fixing the subset of continuous variables, it is easy to find the most violated inequality, see \citet{Dash2016} and \citet{DiSumma2011} for more details.
	Hence it deserves to apply the approach described in this paper to find the combinatorial separation problem of these polyhedra.

	\def\urlprefix{}\def\href#1#2#3#4{\ifstrequal{#2}{[link]}{}{#2\newline}}
	\bibliographystyle{spmpsci}
%	\bibliography{/Users/chenweikun/Documents/Bibtex/divisibleknapsackset.bib}		
	\bibliography{divisibleset.bib}
	
	\end{document}